\documentclass[11pt]{article}

\usepackage{graphicx}
\usepackage{mathptmx}      % use Times fonts if available on your TeX system
\usepackage{amssymb}
\usepackage{amsmath,color}
\usepackage{graphicx, tikz-cd}
\usepackage{yfonts}
\usepackage{amsbsy}
\usepackage{hyperref}
\usepackage{multicol}
\usepackage[margin=1in]{geometry}
\usepackage{listings}
\usepackage{comment} 
\usepackage[utf8]{inputenc}
\usepackage{float}
\usepackage{comment}
\usepackage{amsthm}
\theoremstyle{plain}
\newtheorem{theorem}{Theorem}[section]

\newtheorem{lemma}[theorem]{Lemma}
\newtheorem{question}[theorem]{Question}
\newtheorem{proposition}[theorem]{Proposition}

\theoremstyle{definition}
\newtheorem{assumption}[theorem]{Assumption}
\newtheorem{definition}[theorem]{Definition}
\newtheorem{remark}[theorem]{Remark}
\newtheorem{construction}[theorem]{Construction}

\begin{document}

\title{A Topologically Rigid Set of Quotients of the Davis Complex%\thanks{Grants or other notes
%about the article that should go on the front page should be
%placed here. General acknowledgments should be placed at the end of the article.}
}

\author{Yandi Wu
}

%\date{Received: date / Accepted: date}
% The correct dates will be entered by the editor

\maketitle

\begin{abstract}
A class of topological spaces is \textit{topologically rigid} if any two spaces with the same fundamental group are also homeomorphic. Topological rigidity, in addition to its intrinsic interest, has been useful for solving abstract commensurability questions. In this paper, we explore the topological rigidity of quotients of the Davis complex of certain right angled Coxeter groups by providing conditions on the defining graphs that obstruct topological rigidity. Furthermore, we explore why topological rigidity is hard to achieve for quotients of the Davis complex. Nonetheless, we conclude by introducing infinitely many infinite topologically rigid subclasses. 
%\keywords{topological rigidity \and Davis complex \and right-angled Coxeter group \and orbicomplex}
% \PACS{PACS code1 \and PACS code2 \and more}
%\subclass{20F55 \and 51F15 \and 20F67}
\end{abstract}

\section{Introduction}
\label{intro}
Often, determining whether two topological objects are homeomorphic is a significantly harder problem than determining whether their fundamental groups are isomorphic. In some cases, however, if we impose enough conditions on the topological spaces we are studying, the weaker equivalence relation (isomorphism between fundamental groups) implies the stronger and often more useful equivalence relation (homeomorphism between the topological objects). We can often exploit the topological rigidity of such sets of spaces to derive useful results (recall that a collection of topological objects $\mathcal{X}$ is \textit{topologically rigid} if for any $X_1, X_2 \in \mathcal{X}$, if $\pi_1(X_1) \cong \pi_1(X_2)$, then $X_1$ and $X_2$ are homeomorphic). For example, to determine if two groups $G_1$ and $G_2$ are \textit{abstractly commensurable} (i.e. have isomorphic finite-index subgroups), we often construct two finite-sheeted homeomorphic covers of $X_1$ and $X_2$, where $\pi_1(X_i) = G_i$ for $i = 1, 2$. These homeomorphic covers can be hard to construct. However, if the finite-sheeted covers $\tilde{X_1}$ and $\tilde{X}_2$ belong to a topologically rigid class of spaces and $\pi_1(\tilde{X}_1) \cong \pi_1(\tilde{X_2})$, then we know $\tilde{X}_1$ and $\tilde{X}_2$ are homeomorphic.

There are several well-established examples of topologically rigid classes. For example, the set of closed orientable 2-manifolds is topologically rigid. The Poincare Conjecture implies the set of simply-connected, closed 3-manifolds is topologically rigid. In a series of papers (see \cite{LaFont1}, \cite{LaFont}, and \cite{LaFont2}), Lafont proves the set of of simple, thick $n$-dimensional hyperbolic P-manifolds, a subclass of piecewise CAT(-1) spaces, is topologically rigid for $n \geq 2$. In this paper, we consider certain \textit{orbicomplexes}, unions of collections of orbifolds identified along homeomorphic suborbifolds, associated with Right-Angled Coxeter Groups (RACGs), defined below. 

\begin{definition}[Right-Angled Coxeter Group] Given a finite simplicial graph $\Gamma$ with edge set $E$ and vertex set $V$, the \emph{Right-Angled Coxeter Group (RACG)} $W_{\Gamma}$ with defining graph $\Gamma$ is the group with presentation $\langle v_i \in V : v_i^2 = 1, [v_i, v_j] = 1 \text{ if } [v_i, v_j] \in E\rangle$.
\end{definition} 

A RACG $W_{\Gamma}$ acts properly discontinuously by isometeries on a space called the \textit{Davis Complex} $\Sigma_{\Gamma}$. The quotient $\mathcal{D}_{\Gamma} = \Sigma_{\Gamma}/W_{\Gamma}$, which we call a \textit{Davis orbicomplex}, is one of the aforementioned orbicomplexes and comes equipped with cell stabilizer data defined by the action of $W_{\Gamma}$ on $\Sigma_{\Gamma}$. To clarify, recall that if an amalgamated free product or HNN extension $G$ acts on a Basse Serre tree $T$, the resulting quotient $T/G$ is a graph of groups whose vertices and edges are labelled by subgroups of $G$ isomorphic to vertex and edge stabilizers of $T$. Similarly, each edge and vertex of a Davis orbicomplex $\mathcal{D}_{\Gamma}$ may be labeled by a subgroup of $W_{\Gamma}$ that stabilizes a lift of the edge or vertex in $\Sigma_{\Gamma}$ (in this paper, we do not specify such labels as they are not crucial for our proofs). For further background on the Davis complex and Coxeter groups, refer to \cite{Davis}. The Davis orbicomplex has been studied extensively by Stark, who poses the following question in \cite{Stark}, which we will give a partial answer to in this paper:

\begin{question}
\label{mainq}
For which set $\mathcal{W}$ of Coxeter groups is the set of Davis orbicomplexes $\mathcal{D}_{\Gamma}$ for groups
in $\mathcal{W}$ together with their finite-sheeted covers topologically rigid?
\end{question}

Despite the simplicity of the problem statement, the answer to Question \ref{mainq} is very nuanced. In this paper, we focus our attention on RACGs that are one-ended ($\Gamma$ has no separating edges or vertices and is connected) and hyperbolic ($\Gamma$ is square-free, or has no cycles of length four). One example of a class of defining graphs that gives rise to $W_{\Gamma}$ satisfying these conditions is a subclass of \textit{generalized $\Theta$ graphs}, defined as follows:

\begin{definition}[Generalized $\Theta$-graph]
\label{theta}
For $k \geq 1$, $0 \leq n_1 \leq ... \leq n_k$, let $\Theta = \Theta(n_1, n_2, ..., n_k)$ be the graph with two vertices $a$ and $b$, each of valence $k$, and $k$ edges $e_1, e_2, ..., e_k$ connecting them, which we will call the branches of $\Theta$. Furthermore, for $1 \leq i \leq k$, $e_i$ is subdivided into $n_i + 1$ edges by inserting $n_i$ new vertices. 
\end{definition}

For the purposes of this paper, we will require that $n_i > 0$ for all $1 \leq i \leq k$ and $n_2 > 1$ in order to ensure $W_{\Gamma}$ is hyperbolic. 

Associated to each generalized $\Theta$-graph is an Euler characteristic vector, which captures the Euler characteristics of the orbifolds in the Davis  orbicomplex $\mathcal{D}_{\Gamma}$. The Euler characteristic vector is often used to classify Davis orbicomplexes; in \cite{DST}, the Euler characteristic vector is used to list abstract commensurability criteria. In this paper, we will use Euler characteristic vectors to list criteria for topological rigidity.  

\begin{definition}[Euler characteristic vectors of generalized $\Theta$-graphs]
\label{ECV}
Let $\Theta = \Theta(n_1, n_2, ..., n_k)$ be a generalized $\Theta$ graph. Then the \emph{Euler characteristic vector} of $\Theta$ is the vector $v = (x_1, x_2, ..., x_n)$, where $x_i = \frac{1 - n_i}{4}$. Two Euler characteristic vectors $v_1$ and $v_2$ are said to be \emph{commensurable} if there exist $K, L \in \mathbb{Z}_{\neq 0}$ such that $Kv = Lw$.  
\end{definition}

Dani, Stark, and Thomas show in Theorem 5.2 of \cite{DST} that finite covers of Davis orbicomplexes with $\Gamma = \Theta(n_1, n_2, ..., n_k)$ are topologically rigid. In this paper, we focus on cycles of generalized $\Theta$ graphs introduced in \cite{DST}, which consist of generalized $\Theta$ graphs identified along their essential vertices.

\begin{definition}[Cycle of generalized $\Theta$-graphs]
\label{cycle} 
Let $N \geq 3$ and let $b_1, b_2, ..., b_N$ be positive integers so that for each $i$, $1 \leq i \leq N$, at most one of $b_i$ and $b_{i + 1}$ where $i$ is taken $\mod N$ is equal to 1. Let $\Theta_i$ be a generalized $\Theta$ graph with $b_i$ edges between two vertices $a_i$ and $c_i$. We can construct a cycle of $N$ generalized $\Theta$-graphs $\Gamma$ by identifying $c_i$ with $a_{i + 1}$. 
\end{definition}

We call the vertex of a cycle of generalized $\Theta$ graphs with valence greater than two an \textit{essential vertex}. For the rest of the paper, we will use $\{v_i\}_{i = 1}^{N}$ to denote the set of essential vertices of all the graphs involved. The indices of all $v_i$'s will also taken mod $N$, where $N$ is the number of essential vertices (or equivalently generalized $\Theta$ graphs) in a cycle of generalized $\Theta$-graphs $\Gamma$. 

Let $\Gamma$ be a cycle of generalized $\Theta$ graphs with Davis complex $\Sigma_{\Gamma}$, and $G$ a finite index, torsion-free subgroup of $W_{\Gamma}$. Stark proves in \cite{Stark} that the set of quotients $\Sigma_{\Gamma}/G$, which correspond to finite-sheeted covers of the Davis orbicomplexes $\mathcal{D}_{\Gamma}$, is not topologically rigid by constructing $X_1 = \Sigma_{\Gamma}/G_1$ and $X_2 = \Sigma_{\Gamma}/G_2$ that are homotopic but not homeomorphic. Theorem \ref{main1} in Section 2 generalizes the construction from \cite{Stark} to create a class of orbicomplexes where topological rigidity fails. Our construction of homotopic but not homeomorphic covers relies on the fact that one set of orbifolds in the orbicomplex is a finite cover of another set of orbifolds. 

\begin{definition}
\label{C}
Suppose $\Gamma$ is a cycle of generalized $\Theta$ graphs with essential vertices $\{v_i\}_{i = 1}^{N}$, and there exist two essential vertices $v_j$, $v_k$ such that the generalized $\Theta$ graphs $\Theta_j$ and $\Theta_k$ between $v_j$ and $v_{j + 1} \mod N$ and $v_k$ and $v_{k + 1} \mod N$ (where $k \neq j$) respectively have commensurable Euler characteristic vectors $u$ and $w$ ($Ku = Lw$ for some $K, L \in \mathbb{Z}_{\neq 0}$). Then we say $\Gamma$ is \textit{repetitive}. If $K$ or $L = 1$, then we say $\Gamma$ is \textit{strongly repetitive}. 
\end{definition} 

\begin{theorem}
 \label{main1} 
 Suppose a class of finite-sheeted covers of Davis orbicomplexes $\mathcal{X}$ contains all the finite-sheeted covers of some Davis orbicomplex $\mathcal{D}_{\Gamma}$ where $\Gamma$ is strongly repetitive. Then $\mathcal{X}$ is not topologically rigid. 
 \end{theorem} 
 
 It is not known whether Theorem \ref{main1} is true if we only assume $\Gamma$ is repetitive.

Theorem \ref{main1} as well as Stark's proof in \cite{Stark} rely on constructions of finite-sheeted covers of the same Davis orbicomplex $\mathcal{D}_{\Gamma}$. One can also prove, however, that two finite-sheeted covers of nonhomeomorphic Davis orbicomplexes can also violate topological rigidity.
 
\begin{definition}[Permuted pairs]
\label{Cprime} 
Two cycles of generalized $\Theta$ graphs $\Gamma_1$ and $\Gamma_2$ form a \emph{permuted pair} if they are obtained from identifying the essential vertices of the same set of generalized $\Theta$ graphs. Equivalently, the set of Euler characteristic vectors of $\Gamma_1$ is some permutation of the set of Euler characteristic vectors of $\Gamma_2$.
\end{definition} 

\begin{remark} Note that if we use the definition above, it is possible for a permuted pair $\Gamma_1$ and $\Gamma_2$ to be isomorphic. For example, if $\Gamma_1$ and $\Gamma_2$ each consist of three generalized $\Theta$ graphs glued together, they are isomorphic (see the proof of Lemma \ref{samethetas} for details). We do not consider such pairs in Theorem \ref{main2}, stated below.
\end{remark}

\begin{figure*}[h]
    \centering
    \includegraphics[width=0.6\textwidth]{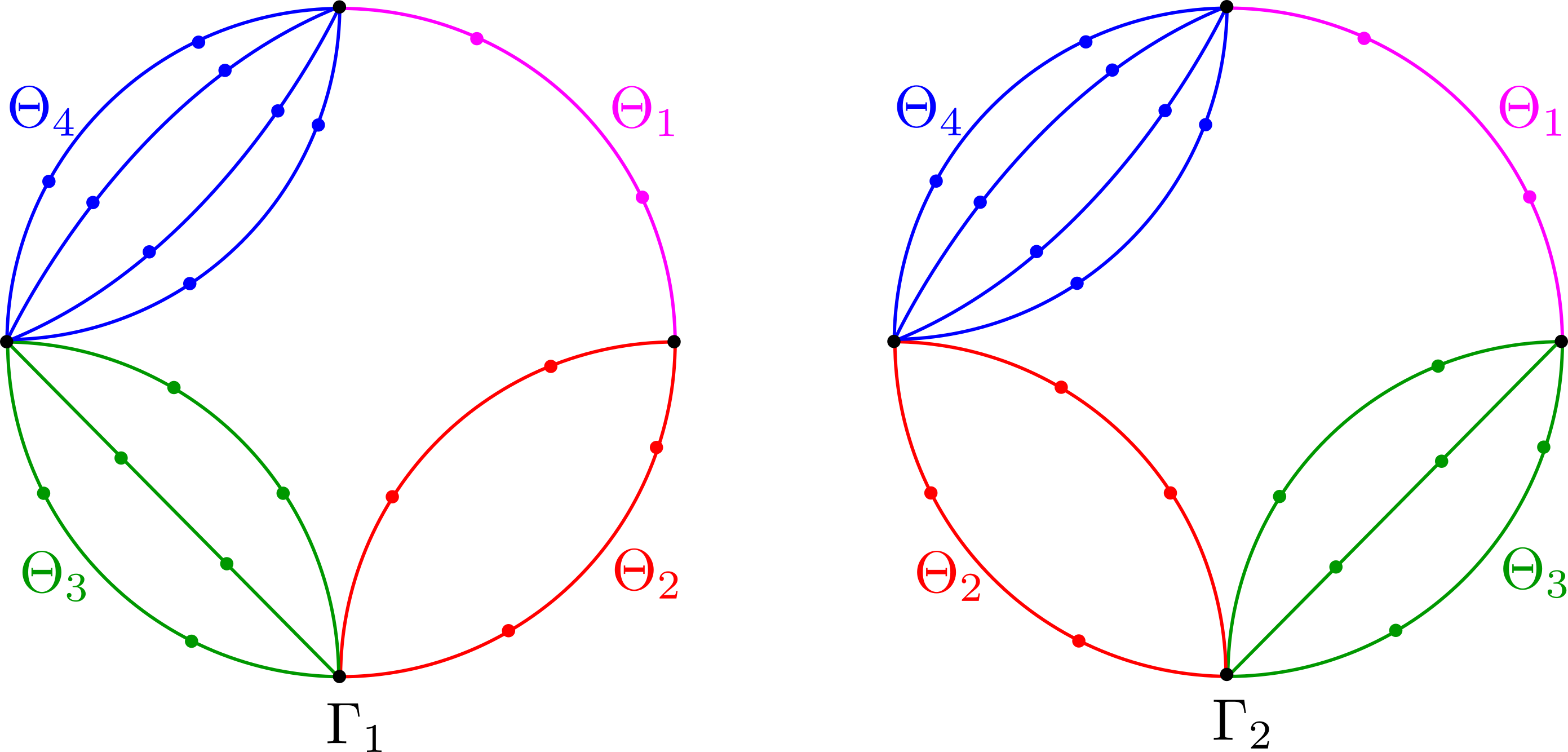}
    \caption{An example of a permuted pair. Note that $\Gamma_1$ and $\Gamma_2$ both consist of $\Theta_i$ (where $1 \leq i \leq 4$) glued along essential vertices.}
    \label{fig:pp}
\end{figure*}
 
 \begin{theorem}
\label{main2} Suppose a class of finite-sheeted covers of Davis orbicomplexes $\mathcal{X}'$ contains all finite sheeted covers of two Davis orbicomplexes $\mathcal{D}_{\Gamma_1}$ and $\mathcal{D}_{\Gamma_2}$, where $\Gamma_1$ and $\Gamma_2$ form a permuted pair and $\Gamma_1$ and $\Gamma_2$ are not isomorphic. Then $\mathcal{X}'$ is not topologically rigid. 
\end{theorem}

In the proof of Theorem \ref{main2}, we find two homotopic finite-sheeted covers of $\mathcal{D}_{\Gamma_1}$ and $\mathcal{D}_{\Gamma_2}$ that are not homemomorphic. As a side note, this means that $W_{\Gamma_1}$ and $W_{\Gamma_2}$ are commensurable, so having two defining graphs that form a permuted pair is a sufficient condition for commensurability. Recall that in Theorem 1.12 of \cite{DST}, Dani, Stark, and Thomas provide two necessary and sufficient conditions for commensurability of RACGs with defining graphs that are cycles of generalized $\Theta$ graphs. 

In \cite{Stark}, Stark constructs $X_1$ and $X_2$, two homotopic finite covers of a Davis orbicomplex $\mathcal{D}_{\Gamma}$ with non-homeomorphic \textit{singular sets} (e.g. sets along which the orbifolds are identified). In her example, $\Gamma$ is a cycle of generalized $\Theta$ graphs, proving that the set of finite-sheeted covers of Davis complexes with defining graphs that are cycles of generalized $\Theta$ graphs is not topologically rigid. In light of these results, in Section 3 of \cite{DST}, Dani, Stark, and Thomas construct a different set of orbicomplexes that is topologically rigid, which they use to prove abstract commensurability results. Nevertheless, in section 4 (see Theorem \ref{main}), we are able to find a topologically rigid subclass of finite-sheeted covers of Davis orbicomplexes $\mathcal{D}_{\Gamma}$ where $\Gamma$ is a cycle of generalized $\Theta$ graphs. The subclass also takes Theorems \ref{main1} and \ref{main2} into account to exclude finite-sheeted covers of $\mathcal{D}_{\Gamma}$ that violate topological rigidity. Although the exact statement of the theorem is fairly technical, we state a simplified version below.

\begin{theorem}
\label{main} 
There exists an infinite class $\mathcal{C}$ of Davis orbicomplexes such that for any $D_{\Gamma} \in \mathcal{C}$, an infinite collection of finite-sheeted covers of $D_{\Gamma}$ form a topologically rigid set. 
\end{theorem} 

\noindent \textbf{Acknowledgments.} The author would like to thank Tullia Dymarz for discussing this work in depth and editing many preliminary drafts. The author would also like to thank the anonymous referee for very helpful comments on a draft of the paper and especially for pointing out edge cases in the proofs of Theorems \ref{main1} and \ref{main2}. Additionally, the author thanks Emily Stark for helpful discussions.

%Separate text sections with
\section{Preliminaries}
\label{sec:2}

We now introduce a construction of the Davis orbicomplex specific to the setting where the defining graph $\Gamma$ is a cycle of generalized $\Theta$ graphs consisting of $\Theta_i = \Theta(n_{i,1}, n_{i,2},...,n_{i,k})$ for $1 \leq i \leq N$. For a more detailed construction of $W_{\Gamma}$ and verification that $\pi_1(\mathcal{D}_{\Gamma})$ is indeed $W_{\Gamma}$, we refer the reader to Section 2 of \cite{Stark} and Section 3 of \cite{DST}.

First, we describe how to construct an orbifold $\mathcal{P}_{i,j}$ for a branch (edge) $b_{i,j}$ of a generalized $\Theta$ graph $\Theta_i$. For each $b_{i,j}$, construct a $(n_{i,j} + 2)$-gon with an edge of order 1, which we call a \textit{nonreflection edge}, and $n_{i,j} + 1$ reflection edges of order 2. All the vertices are order 4 vertices, with the exception of the two order 2 vertices adjacent to the non-reflection edge.

\begin{construction}[Davis Orbicomplex $\mathcal{D}_{\Gamma}$ of a cycle of generalized $\Theta$ graphs] 
\label{orbicomplexcons} First, we will construct an orbifold graph $S$. The underlying graph of $S$ is a star with one central vertex $v_0$ adjacent to $N$ valence one vertices. The valence one vertices are orbifold points of order 2. Cyclically label the orbifold points with $v_l$ where $1 \leq l \leq N$, and use $e_l$ to denote the edge $[v_0, v_l] \in E(S)$. Then attach the set of branch orbifolds $\mathcal{P}_{i,j}$ along their non-reflection edges to $e_i$ and $e_{i + 1}$, where the labels are taken mod $N$. An example of Construction \ref{orbicomplexcons} is shown in Figure \ref{fig:orbicomplex}.
o\end{construction}  

\begin{figure}[h!]
    \centering
    \includegraphics[width=0.8\textwidth]{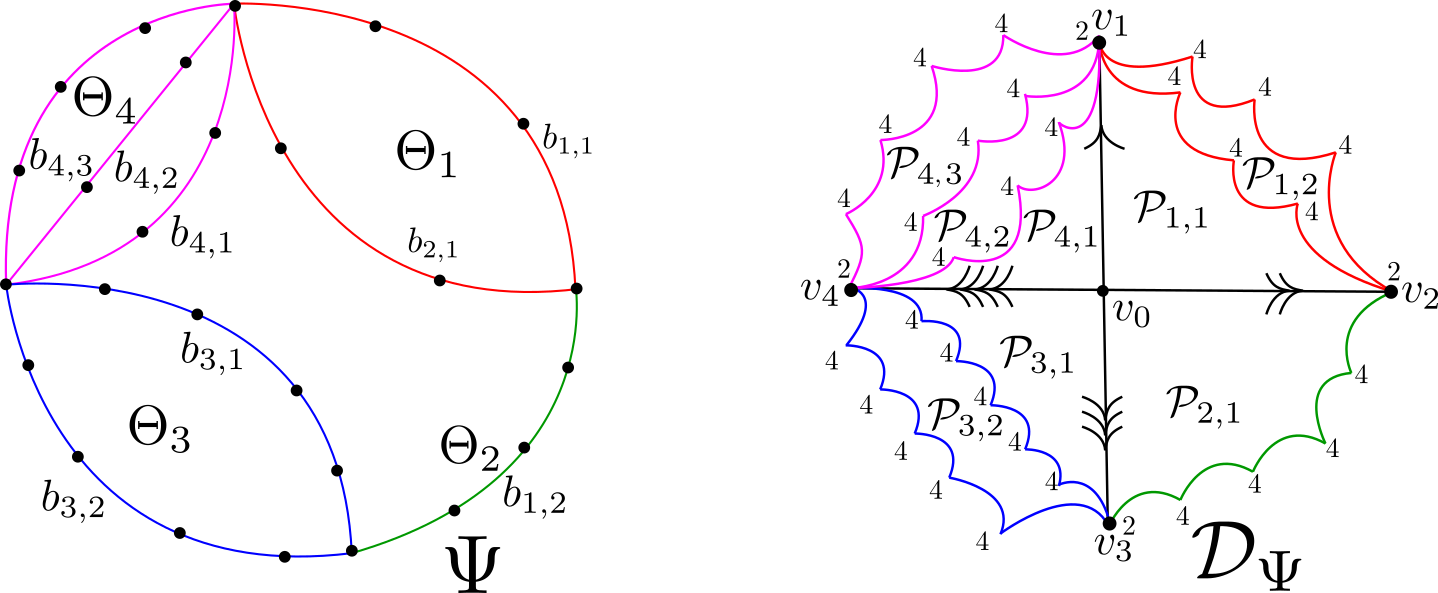}
    \caption{A cycle of generalized $\Theta$ graphs $\Psi$ along with its Davis orbicomplex $\mathcal{D}_{\Psi}$. We label edges $e_l$ in the singular star $S$ with $l$ arrows. Note that each branch $b_{i,j} \in \Psi$ determines an orbifold $\mathcal{P}_{i,j} \in \mathcal{D}_{\psi}$.}
    \label{fig:orbicomplex}
\end{figure}

Note that for the cycle of generalized $\Theta$ graphs $\Psi$ shown in Figure \ref{fig:orbicomplex}, the Euler characteristic vectors of $\Theta_1$ and $\Theta_3$ are $(-\frac{1}{4}, -
\frac{1}{4})$ and $(-\frac{3}{4}, -\frac{3}{4})$, so $3w_1 = w_3$, which means $\Psi$ is strongly repetitive. Theorem 1.7 then implies any class $\mathcal{X}$ that contains all finite-sheeted covers of $\mathcal{D}_{\Psi}$ is not topologically rigid.

All finite-sheeted covers of Davis orbicomplexes that we construct will contain \textit{jester hats}, a specific kind of orbifold defined below:

\begin{definition}[Jester hats]
Suppose $\mathcal{O} = D^2\underbrace{(2, 2, ..., 2)}_{n}$, i.e. a disk with $n$ order 2 points. Then we will call $\mathcal{O}$ a jester hat with $n$ (order two) cone points. 
\end{definition}

\section{Examples of topologically non-rigid sets}
\label{sec:3}

We first introduce the construction of jester hats that cover orbifolds in a Davis orbicomplex $\mathcal{D}_{\Gamma}$. 

\begin{lemma} 
\label{jh} 
Suppose $d$ is a positive even integer. Each orbifold $\mathcal{P}$ with $r$ reflection edges is covered by the jester hat $D^2\underbrace{(2, 2, ... ,2)}_{c}$, where $c = \frac{d}{2}(r - 3) + 2$ and $d$ is the degree of the cover.   
\end{lemma} 

\begin{proof} This construction is based on Stark's construction in Lemma 3.1 from \cite{Stark} and Crisp and Paoluzzi's construction from Section 3.1 of \cite{cp}. Using Crisp and Paoluzzi's construction, we observe that for any even integer $d > 0$, an orbifold $\widehat{\mathcal{O}}$ with $\frac{d}{2}(r - 3) + 3$ reflection edges is tiled by $\frac{d}{2}$ copies of an orbifold with $r$ reflection edges, so $\widehat{\mathcal{O}}$ is a $\frac{d}{2}$-sheeted orbifold cover of $\mathcal{O}$. For example, in Figure \ref{fig:fincovs}, an orbifold $\widehat{\mathcal{O}}$ with $6$ reflection edges is tiled by 3 copies of $\mathcal{O}$, an orbifold with $4$ reflection edges. Thus, $\widehat{\mathcal{O}}$ is a 3-sheeted orbifold cover of $\mathcal{O}$. Next, if we unfold along the reflection edges of $\widehat{\mathcal{O}}$, we obtain a closed disk with $\frac{d}{2}(r - 3) + 2$ order two cone points, as desired. The construction is illustrated in Figure \ref{fig:fincovs}. 

%Suppose $v_a$ and $v_b$ are two vertices adjacent to the non-reflection edge of $\mathcal{P}$. Let $r_a$ and $r_b$ denote the reflection edges adjacent to $v_a$ and $v_b$ respectively. We can then unfold along $r_a$ so that the nonreflection edge forms a geodesic path, and then unfold along $r_b$. If we repeat this process $\frac{d}{2}$ times, alternating between unfolding along $r_a$ and $r_b$, we obtain an orbicomplex with $\frac{d}{2}(r - 3) + 3$ reflection edges.   
\end{proof} 

\begin{figure}[H] 
    \begin{center} 
    \includegraphics[width=\textwidth]{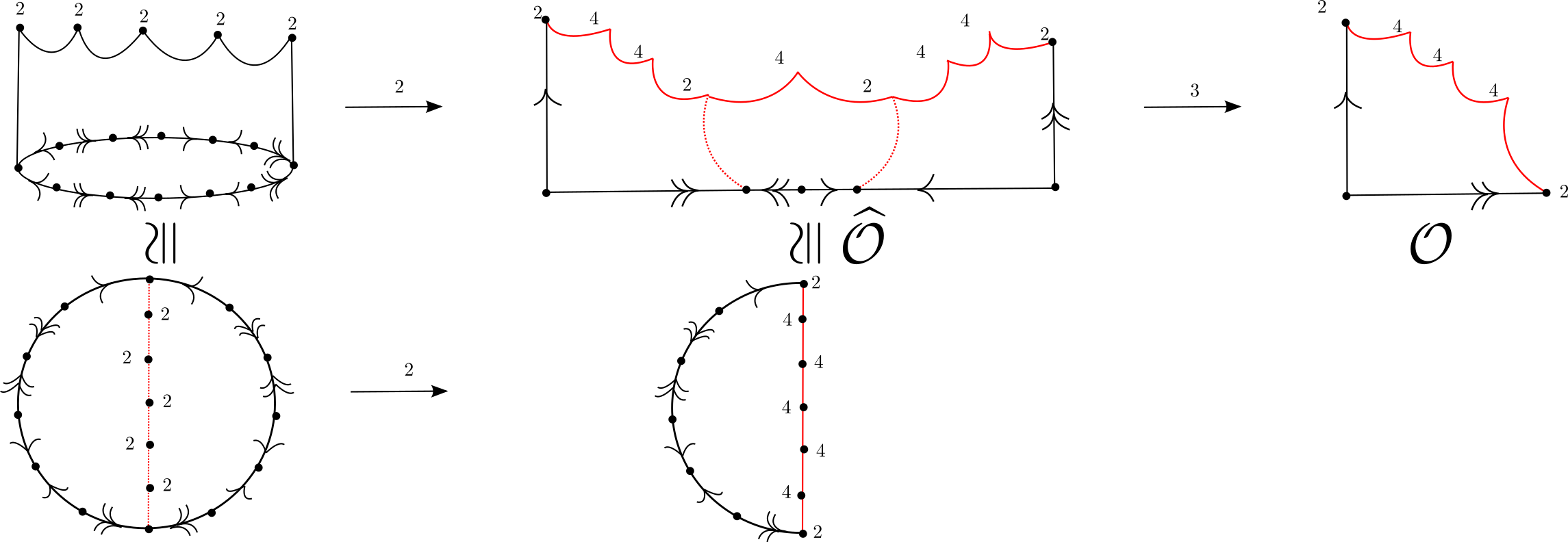}
    \end{center} 
    \label{fig:fincovs}
    \caption{A tower of covers illustrating the lemma. Here, $D^2(2, 2, 2, 2, 2)$ is a six-fold cover of an orbifold with four reflection edges.} 
\end{figure} 

\begin{construction}[The double of a singular set]
\label{double} 
The Davis orbicomplex $\mathcal{D}_{\Gamma}$ has a singular subset $S$ consisting of an orbifold star graph with $N$ order two points, where $N$ is the number of essential vertices in $\Gamma$. We can construct a double cover of $S$, which we call $\widehat{S}$, by unfolding along order two points to obtain a subdivided generalized $\Theta$-graph with $N$ branches and one vertex on each branch between the essential vertices. For an illustration, refer to Figure \ref{fig:thetacover}. All the singular sets constructed in this section will be a finite-sheeted cover of $\widehat{S}$. 

\begin{figure}[h]
    \centering
    \includegraphics[width=0.6\textwidth]{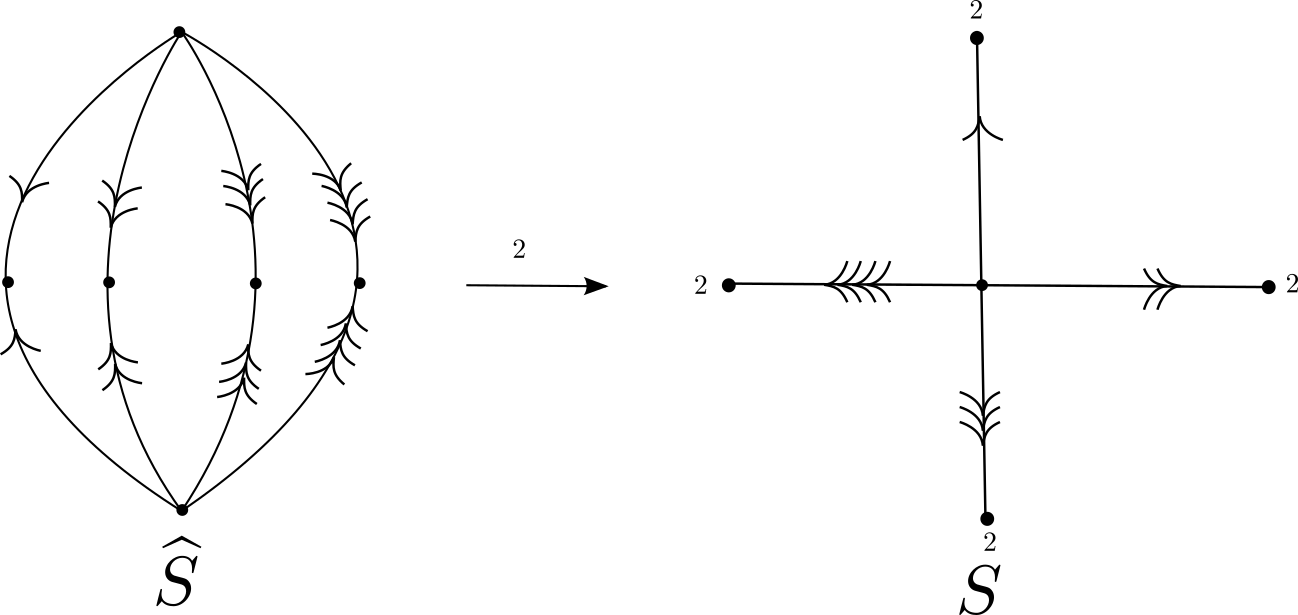}
    \caption{A generalized $\Theta$ graph with four branches two-fold covers the singular subset of the orbicomplex from Figure \ref{fig:orbicomplex}. Here, $N = 4$.}
    \label{fig:thetacover}
\end{figure}
\end{construction} 

For all of the covers described in this paper, all the edges will be subdivided by a copy of the lift of an order two point $\tilde{v_i}$ from the Davis orbicomplex. Both subdivisions will be oriented towards $\tilde{v_i}$ and labeled with the label of the edge, which we will specify in the construction. For simplicity, we will count a subdivided edge as one edge when calculating cycle lengths.

\begin{proposition}
Let $\Gamma$ be strongly repetitive. Then there exist homotopic but non-homeomorphic finite-sheeted covers of $\mathcal{D}_{\Gamma}$.
\end{proposition}

\noindent \textit{Proof.} Suppose $u$ and $w$ are Euler characteristic vectors of $\Theta_i$ and $\Theta_k$ respectively where $Ku = w$ for some $K \in \mathbb{Z}_{+}$. Without loss of generality, assume $i = 1$ since we can rotate the labels of the essential vertices otherwise. Suppose $\Theta_1$ and $\Theta_k$ have $l$ branches. Let $n_{s,b}$ denote the number of vertices on the $b$th branch of $\Theta_s$. Then if $u = (\frac{1 - n_{1,1}}{4}, \frac{1 - n_{1,2}}{4}, ..., \frac{1 - n_{1,l}}{4})$ and $w = (\frac{1 - n_{k,1}}{4}, \frac{1 - n_{k,2}}{4}, ..., \frac{1 - n_{k,l}}{4})$ then for $1 \leq b \leq l$, $K\big(\frac{1 - n_{1,b}}{4}\big) = \big(\frac{1 - n_{k,b}}{4}\big)$ and $K(1 - n_{1,b}) = 1 - n_{k,b}$. If $r_{s,b}$ denotes the number of reflection edges on the orbifold in the Davis orbicomplex constructed from the $b$th branch of $\Theta_s$, then $r_{s,b} = n_{s,b} + 2$, so we have $K(r_{1,b} - 3) + 3 = r_{k,b}$. \\

\noindent \textbf{Case 1} ($\mathbf{K = 1}$):
We will first consider some general cases before addressing the edge case where $N \neq 3$. First, suppose $K = 1$ and $k < N - 1$. Then $r_{1, b} = r_{k, b}$ for $1 \leq b \leq l$. We will construct two non-homeomorphic but homotopic four-sheeted covers of $\mathcal{D}_{\Gamma}$, which we will call $\widetilde{X_1}$ and $\widetilde{X_2}$. First, we construct their singular sets $\widetilde{S_1}$ and $\widetilde{S_2}$ with four essential vertices, $\tilde{v}_1, \tilde{v}_2, \tilde{v}_3, \tilde{v}_4$ and $\tilde{v}'_1, \tilde{v}'_2, \tilde{v}'_3, \tilde{v}'_4$.  To construct $\widetilde{S_1}$, add $k$ edges between two pairs of vertices: $\tilde{v}_1$ and $\tilde{v}_4$, as well as $\tilde{v}_2$ and $\tilde{v}_3$. The edges will be labeled with all integers between $1$ and $k$ and subdivided as described earlier in the section. Between two other pairs of vertices, $\tilde{v}_{1}$ and $\tilde{v}_2$, as well as $\tilde{v}_3$ and $\tilde{v}_4$, construct $N - k$ subdivided edges labeled with all integers between $k + 1$ and $N$. Thus, in total, there are two (subdivided) four-cycles labeled with $k$ and $k + 1$ as well as $1$ and $N$, and $N - 2$ cycles of length two labeled with $i$ and $i + 1$, where $1 \leq i \leq N - 1$, $i \neq k$. For the other singular set, $\widetilde{S_2}$, there is one edge labeled with $1$ between two pairs of vertices: $\tilde{v}'_1$ and $\tilde{v}'_4$, as well as $\tilde{v}'_2$ and $\tilde{v}'_3$. Additionally, there are $N - 1$ edges between two other pairs of vertices, $\tilde{v}'_1$ and $\tilde{v}'_2$ as well as $\tilde{v}'_3$ and $\tilde{v}'_4$; these edges are labeled with all integers between $2$ and $N$. In total, there are again two four-cycles with edges labeled with $1$ and $2$ as well as $1$ and $N$, and $N - 2$ cycles of length two labeled with $i$ and $i + 1$ where $2 \leq i \leq N - 1$. 

By the conditions imposed on $\mathcal{D}_{\Gamma}$, the four-cycles labeled by $k$ and $k + 1$ in $\widetilde{S_1}$ and $1$ and $2$ in $\widetilde{S_2}$ have the same set of jester hats glued to them. Additionally, the two-cycles labeled with $1$ and $2$ in $\widetilde{S_1}$ and $k$ and $k + 1$ in $\widetilde{S_2}$ have the same set of jester hats glued to them. For other integers $I$ where $1 \leq I \leq N$ and $I \neq 1, k$, for every cycle in $\widetilde{S_1}$ labeled with $I$ and $I + 1$, there is a corresponding cycle of the same length labeled with $I$ and $I + 1$ in $\widetilde{S_2}$. As a result, for each cycle in $\widetilde{S_1}$ with a set of jester hats glued to it, there is a corresponding cycle in $\widetilde{S_2}$ with the same set of jester hats glued to it. By taking their regular neighborhoods, we can see that both $\widetilde{S_1}$ and $\widetilde{S_2}$ are homotopic to a $(2N - 4)$-holed sphere. We can then conclude that $\widetilde{X_1}$ and $\widetilde{X_2}$ are homotopic to the same $(2N - 4)$-holed sphere with the same same sets of jester hats glued to their boundary components. However, they are not homeomorphic since the complements of their cut pairs have different numbers of connected components.
%Added edge cases

If $N \neq 3$, there are two special cases where the above construction does not work. First, note that if $k = N$, the construction does not work because the first cover will be disconnected. Thus, we construct modified non-homeomorphic but homotopic degree four covers. As before, we will call these covers $\widetilde{X_1}$ and $\widetilde{X_2}$ with singular sets $\widetilde{S_1}$ and $\widetilde{S_2}$ that have essential vertices $\tilde{v_i}$ and $\tilde{v'_i}$ respectively, where $1 \leq i \leq 4$. For $\widetilde{S_1}$, construct two subdivided edges labeled with 1 and 2 between $\tilde{v_1}$ and $\tilde{v_4}$ as well as $\tilde{v_2}$ and $\tilde{v_3}$. Then between vertices $\tilde{v_3}$ and $\tilde{v_4}$ as well as $\tilde{v_1}$ and $\tilde{v_2}$, construct $N - 2$ edges with labels between $3$ and $N$. For $\widetilde{S_2}$, between $\tilde{v'_1}$ and $\tilde{v'_4}$ as well as $\tilde{v_2}$' and $\tilde{v'_3}$, construct $N - 1$ subdivided edges labeled with all integers between $3$ and $N$ and one subdivided edge labeled with 1. Then construct a subdivided vertex labelled with 2 between $\tilde{v'_3}$ and $\tilde{v'_4}$ as well as $\tilde{v'_1}$ and $\tilde{v'_2}$.  

Second, note that if $k = N - 1$, the general construction does not work since $\widetilde{S_1}$ and $\widetilde{S_2}$ are homeomorphic; therefore, $\widetilde{X_1}$ and $\widetilde{X_2}$ are also homeomorphic. In order to fix this, construct new $\widetilde{S_1}$ and $\widetilde{S_2}$ as follows: for $\widetilde{S_1}$, construct 3 edges between $\tilde{v_1}$ and $\tilde{v_2}$ as well as $\tilde{v_3}$ and $\tilde{v_4}$, which are labeled with $1, 2$, and $N$. Then construct $N - 3$ edges between $\widetilde{v_1'}$ and $\widetilde{v'_4}$ as well as $\widetilde{v'_2}$ and $\widetilde{v'_3}$ labeled from $3$ to $N - 1$. For $\widetilde{S_2}$, follow the construction for the case where $k = N$. Again, in both edge cases ($k = N$, $k = N - 1$, $N > 3$), both $\widetilde{S_1}$ and $\widetilde{S_2}$ are homotopic to a $(2N - 4)$-holed sphere and have the same sets of jester hats glued to them, but complements of their cut pairs have different numbers of connected components. Thus, $\widetilde{S_2}$ and $\widetilde{S_1}$ are not homeomorphic but $\widetilde{X_1}$ and $\widetilde{X_2}$ are homotopic.

For the case $N = 3$, any of the previous constructions will yield homeomorphic $\widetilde{X_1}$ and $\widetilde{X_2}$, so a different pair of covers is necessary. For this special case, we will construct 16-sheeted covers. Assume without loss of generality that $\Theta_1$ and $\Theta_3$ have the same Euler characteristic vectors. For both $\widetilde{S_1}$ and $\widetilde{S_2}$, label edges $[v_i, v_{i + 1}]$ and $[v'_i, v'_{i + 1}]$ respectively with $2$ if $i$ is odd and $3$ if $i$ is even. Then construct the edges $[\widetilde{v_2}, \widetilde{v_{13}}]$, $[\widetilde{v_3}, \widetilde{v_{12}}]$, $[\widetilde{v_4}, \widetilde{v_{11}}]$, $[\widetilde{v_1}, \widetilde{v_{16}}]$, $[\widetilde{v_{14}}, \widetilde{v_{15}}]$, $[\widetilde{v_9}, \widetilde{v_{10}}]$, $\widetilde{v_7}, \widetilde{v_8}]$, and $[\widetilde{v_5}, \widetilde{v_6}]$ labelled with $1$. For $\widetilde{S_2}$, construct edges $[\widetilde{v'_1}, \widetilde{v'_{16}}]$, $[\widetilde{v'_2}, \widetilde{v'_{15}}]$, $[\widetilde{v'_{13}}, \widetilde{v'_{14}}]$, $[\widetilde{v'_3}, \widetilde{v'_{12}}]$, $[\widetilde{v'_4}, \widetilde{v'_5}]$, $[\widetilde{v'_6}, \widetilde{v'_7}]$, $[\widetilde{v'_8}, \widetilde{v'_{11}}]$, and $[\widetilde{v'_9}, \widetilde{v'_{10}}]$ and label them with $1$. As a result, $\widetilde{S_1}$ has one 6-cycle, one 4-cycle, and three 2-cycles labelled with 1 and 2 and one 8-cycle, one 4-cycle, and two 2-cycles labelled with 1 and 3. On the other hand, $\widetilde{S_2}$ has the same set of cycle counts with different labels: one 6-cycle, one 4-cycle, and three 2-cycles labelled with 1 and 3 and one 8-cycle, one 4-cycle, and two 2-cycles labelled with 1 and 2. Thus, both $\widetilde{S_1}$ and $\widetilde{S_2}$ are homotopic to $10$-holed spheres, and by construction have the same sets of jester hats glued to them; thus, $\widetilde{X_1}$ and $\widetilde{X_2}$ are homotopic. Observe, however, that $\widetilde{S_1}$ and $\widetilde{S_2}$ are not homeomorphic, as desired. This completes the proof for $K = 1$. 

\noindent \textbf{Case 2} ($\mathbf{K > 1}$): Suppose $K > 1$ and $K = p_1^{u_1}p_2^{u_2}...p_t^{u_t} = \prod\limits_{d = 1}^t p_d^{u_d}$ is the prime factorization of $K$. We will now construct two non-homeomorphic $2\prod\limits_{d = 1}^t p_d^{u_d + 1}$-sheeted covers $\widetilde{X_1}$ and $\widetilde{X_2}$. First, we construct their singular sets $\widetilde{S_1}$ and $\widetilde{S_2}$, which will both have $2\prod\limits_{d = 1}^n p_d^{u_d + 1}$ essential vertices, which we will label $\tilde{v}_i$ and $\widetilde{v'}_i$ for $1 \leq i \leq 2\prod\limits_{d = 1}^t p_d^{u_d + 1}$. Note that all the vertex indices in the construction will be taken modulo $2\prod\limits_{d = 1}^t p_d^{u_d}$.\\

\noindent 
\textbf{Construction of $\widetilde{S_1}$:} First, suppose that $k \neq N$; if $k = N$, the following construction will be disconnected. For the first singular set $\widetilde{S_1} \subset \widetilde{X_1}$, for even $i$, construct $k$ edges between $\tilde{v}_i$ and $\tilde{v}_{i + 1}$ mod $2\prod\limits_{d = 1}^t p_d^{u_d + 1}$. The edges will be labeled with all integers $I$ where $2 \leq I \leq k + 1$. For odd $i$, construct $N - k$ edges between $\tilde{v}_i$ and $\tilde{v}_{i + 1} \mod 2\prod\limits_{d = 1}^t p_d^{u_d + 1}$. One of these edges will be labeled with $1$ and the rest with all integers $I$ where $k + 2 \leq I \leq N$. As a result, for $i \neq 1, k + 1$, between $\widetilde{v_i}$ and $\widetilde{v_{i + 1}}$, there is a copy of a two-cycle that is a double cover of the nonreflection edges in $\mathcal{D}_{\Gamma}$ labeled with $i$ and $(i + 1)$ $\mod N$. We will then have a graph with $\prod\limits_{d = 1}^t p_d^{u_d + 1}$ copies of two-cycles labeled with $i$ and $(i + 1)$ for $2 \leq i \leq N, i \neq k + 1$, as well as two $2\prod\limits_{d = 1}^t p_d^{u_d + 1}$-cycles, one labeled with $1$ and $2$, and one labeled with $k + 1$ and $k + 2$. For an example, refer to the graph on the right in Figure \ref{fig:homcovers}, which is an 18-sheeted cover of the singular set from the orbicomplex in Figure \ref{fig:orbicomplex}.\\ 

We now consider the special case where $k = N$. For even $i$, construct $N - 2$ edges between $\tilde{v_i}$ and $\tilde{v}_{i + 1}$ labeled with integers $I$ where $2 \leq I \leq N - 1$. Then for odd $i$, construct $2$ edges between vertices $\tilde{v_i}$ and $\tilde{v}_{i + 1}$ labeled with $1$ and $N$. We will again have a graph with $\prod\limits_{d = 1}^t p_d^{u_d + 1}$ copies of two-cycles labeled with $i$ and $(i + 1)$ for $ i = N$ and $2 \leq i \leq N - 2$. We will also still have two $2\prod\limits_{d = 1}^t p_d^{u_d + 1}$-cycles, one labeled with $N - 1$ and $N$ and the other with $1$ and $2$ as before. 

\noindent 
\textbf{Construction of $\widetilde{S_2}$:} For the second singular set $\widetilde{S_2} \subset \widetilde{X_2}$, first partition $\{\tilde{v'}_i\}$ into $p = \prod\limits_{d = 1}^t p_d$ sets of equal size, so the first set will contain vertices $\tilde{v'}_i$ where $1 \leq i \leq 2\prod\limits_{d = 1}^t p_d^{u_d}$, the second set will contain vertices where $2\prod\limits_{d = 1}^t p_d^{u_d} + 1 \leq i \leq 4\prod\limits_{d = 1}^t p_d^{u_d}$, so in general, the $n$th set will contain $\tilde{v'}_i$ labeled $2n\bigg(\prod\limits_{d = 1}^t p_d^{u_d}\bigg) + 1 \leq i \leq 2(n + 1)\prod\limits_{d = 1}^t p_d^{u_d}$, where $0 \leq n \leq p - 1$. 

We first consider the case where $k \neq N$. Construct $k - 1$ edges between the pairs of vertices labeled $\tilde{v'}_{2n\big(\prod\limits_{d = 1}^t p_d^{u_d}\big) + 1}$ and $\tilde{v'}_{2(n + 1)\prod\limits_{d = 1}^t p_d^{u_d}}$, for $0 \leq n \leq p - 1$. Label these edges with integers $I$, where $2 \leq I \leq k$. For all $1 \leq n \leq p$, construct an edge each labeled with $k + 1$ between $\tilde{v'}_{2n\prod\limits_{d = 1}^t p_d^{u_d}}$ and $\tilde{v'}_{2n\big(\prod\limits_{d = 1}^t p_d^{u_d}\big) + 1}$. Finally, add edges to the remaining vertices with no edges between them. For even $i$, if there are no edges between $\tilde{v'}_i$ and $\tilde{v'}_{i + 1}$, add $k$ edges and label them with integers $I$ where $2 \leq I \leq k + 1$. For odd $i$, add $k$ edges between all $\tilde{v'}_i$ and $\tilde{v'}_{i + 1}$ with no edges between them. One of these edges will be labeled with a $1$, while the rest are labeled with integers $I$ such that $k + 2 \leq I \leq N$. As a result, we will have $p$ cycles of length $2\prod\limits_{d = 1}^t p_d^{u_d}$ labeled with 1 and 2, $\prod\limits_{d = 1}^t p_d^{u_d + 1}$ two-cycles labeled with $i$ and $i + 1$ for $2 \leq i \leq k$ or $k + 2 \leq i \leq N$, $\prod\limits_{d = 1}^t p_d^{u_d}$ two-cycles labeled with $k$ and $k + 1$, and a $2p$-cycle labeled with $k$ and $k + 1$. For an example, see the graph on the left in Figure \ref{fig:homcovers}. \\

Now consider the edge case of $k = N$. This time, we will construct one edge labeled with $1$ between the pairs of vertices labeled $\tilde{v'}_{2n\big(\prod\limits_{d = 1}^t p_d^{u_d}\big) + 1}$ and $\tilde{v'}_{2(n + 1)\prod\limits_{d = 1}^t p_d^{u_d}}$, for $0 \leq n \leq p - 1$. For all $1 \leq n \leq p$, construct an edge labeled with $N$ between $\tilde{v'}_{2n\prod\limits_{d = 1}^t p_d^{u_d}}$ and $\tilde{v'}_{2n\big(\prod\limits_{d = 1}^t p_d^{u_d}\big) + 1}$. For the other pairs of vertices, for even $i$, if there are no edges between $\tilde{v'}_i$ and $\tilde{v'}_{i + 1}$, add $2$ edges and label them with $1$ and $N$. For odd $i$, add $N - 2$ edges between all $\tilde{v'}_i$ and $\tilde{v'}_{i + 1}$ labeled with integers $I$ such that $2 \leq I \leq N - 1$. As a result, we will have $p$ cycles of length $2\prod\limits_{d = 1}^t p_d^{u_d}$ labeled with 1 and 2, $\prod\limits_{d = 1}^t p_d^{u_d + 1}$ each of two-cycles labeled with $i$ and $i + 1$ for $2 \leq i \leq N - 2$ or $i = N$, and a $2p$-cycle labeled with $1$ and $N$.

Observe that $\widetilde{S_1}$ and $\widetilde{S_2}$ are not homeomorphic since $\widetilde{S_1}$ has many more cut pairs. In particular, any two essential vertices will form a cut pair in $\widetilde{S_1}$, but $\widetilde{S_2}$ only has $2p \choose 2$ pairs of cut pairs. As a result, $\widetilde{X}_1$ and $\widetilde{X}_2$ are not homeomorphic. However, note that $\widetilde{S_1}$ and $\widetilde{S_2}$ are homotopic; take a regular neighborhood of both graphs to obtain a $2\prod\limits_{d = 1}^t p_d^{u_d + 1}(N - 2) + 1$-holed sphere. 

\begin{figure}
    \centering
    \includegraphics[width=0.8\textwidth]{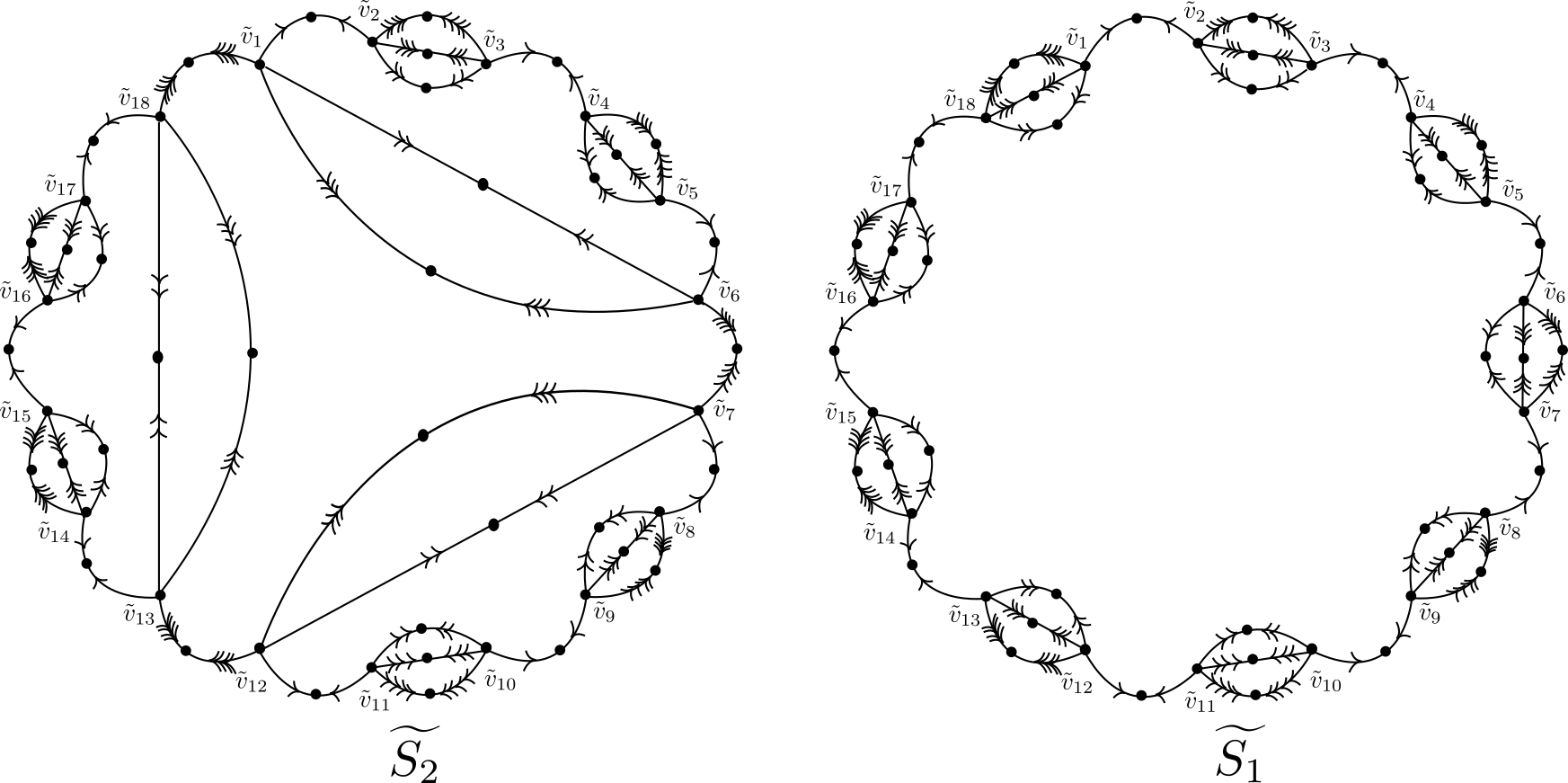}
    \caption{Two homotopic but non-homeomorphic covers of the singular set of $W_{\Gamma}$ from Figure \ref{fig:orbicomplex}. Instead of labeling the edges with numbers, we use different numbers of arrows.}
    \label{fig:homcovers}
\end{figure}

We then examine what orbifolds are glued to $\widetilde{S_1}$. Recall that $r_{s, b}$ denotes the number of reflection edges of the orbifold corresponding to the $b$th branch of $\Theta_s \subset \Gamma$. By Lemma 2.1, we can see that there will be a disk with $\prod\limits_{d = 1}^t p_d^{u_d + 1}(r_{s,b} - 3) + 2$ cone points glued along its boundary circle to the $2\prod\limits_{d = 1}^t p_d^{u_d}$-cycles labeled by $1$ and $2$ and $k + 1$ and $k + 2$ for $s = 1, k + 1$ for the case $k \neq N$; if $k = N$, the second $2\prod\limits_{d = 1}^t p_d^{u_d}$-cycle is labeled with $N - 1$ and $N$ and $s = N - 1$. Additionally, there will be a disk with $r_{i,b} - 1$ cone points glued to each of the $\prod\limits_{d = 1}^t p_d^{u_d + 1}$ two-cycles labeled by $i$ and $i + 1 \mod N + 1$ for $2 \leq i \leq k$ or $k + 2 \leq i \leq N$ if $k \neq N$ ($2 \leq i \leq N - 2$ or $i = N$ if $k = N$). 

Next, we list the orbifolds glued to $\widetilde{S_2}$. First, there will be $p$ copies of disks with $\prod\limits_{d = 1}^t p_d^{u_d}(r_{1,b} - 3) + 2$ cone points glued to the $2\prod\limits_{d = 1}^t p_d^{u_d}$-cycles with edges labeled with 1 and 2. Second, there will be one copy of a disk with $\prod\limits_{d = 1}^t p_d^{u_d + 1}(r_{k + 1, b} - 3) + 2$ cone points glued to the $2\prod\limits_{d = 1}^t p_d^{u_d + 1}$ cycle labeled by $k + 1$ and $k + 2$ if $k \neq N$ and $N - 1$ and $N$ if $k = N$. There will also be $\prod\limits_{d = 1}^t p_d^{u_d + 1}$ copies of a disk with $r_{i,b} - 1$ cone points glued to each 2-cycle labeled by $i$ and $i + 1 \mod N$ for $2 \leq i \leq k - 1$ or $k + 1 \leq i \leq N$ if $k \neq N$ and $2 \leq i \leq N - 2$ if $k = N$. Finally, there are $\prod\limits_{d = 1}^t p_d^{u_d}$ copies of a disk with $r_{k,b} - 1$ cone points glued to each two-cycle labeled by $k$ and $k + 1$  as well as one copy of a disk with $p(r_{k,b} - 3) + 2$ cone points glued to the $2p$-cycle labeled by $k$ and $k + 1$. Then for both orbicomplex covers, we have the same collection of orbifolds glued to $\widetilde{S_1}$ and $\widetilde{S_2}$: 

\begin{itemize}
    \item $\prod\limits_{d = 1}^t p_d^{u_d + 1}$ copies of each disk with $r_{i,b} - 1$ cone points, where $2 \leq i \leq k - 1$ or $k + 1 \leq i \leq N$ if $k \neq N$ and $2 \leq i \leq N - 2$ if $k = N$; 
    \item One copy of a disk with $\prod\limits_{d = 1}^t p_d^{u_d + 1}(r_{k + 1, b} - 3) + 2$ cone points;
    \item $\prod\limits_{d = 1}^t p_d^{u_d} + p = \prod\limits_{d = 1}^t p_d^{u_d + 1}$ copies of a disk with $r_{k, b} - 1 = \prod\limits_{d = 1}^t p_d^{u_d}(r_{1,b} - 3) + 2$ cone points since there are $\prod\limits_{d = 1}^t p_d^{u_d}$ copies of two-cycles labeled with $k$ and $k + 1$ as well as $p$ copies of $2\prod\limits_{d = 1}p_d^{u_d}$-cycles labeled with $1$ and $2$ in $\widetilde{S_2}$ and $\prod\limits_{d = 1}^t p_d^{u_d + 1}$ copies of two-cycles labeled with $k$ and $k + 1$ in $\widetilde{S_1}$ ;  
    \item One copy of the disk with $p(r_{k,b} - 3) + 2 = \prod\limits_{d = 1}^t p_d^{u_d + 1}(r_{1, b} - 3) + 2$ cone points since there is one 2$p$-cycle labeled with $k$ and $k + 1$ in $\widetilde{S_2}$ and one 2$\prod\limits_{d = 1}^t p_d^{u_d + 1}$-cycle labeled with 1 and 2 in $\widetilde{S_1}$.   
\end{itemize}
 As a result, since $\widetilde{X}_1$ and $\widetilde{X}_2$ have the same sets of jester hats glued to homotopic graphs, we have found finite covers that are homotopic but not homeomorphic. This proves Theorem \ref{main1}.\\ \qed 

Next, we prove a Lemma which immediately implies Theorem \ref{main2}.

\begin{lemma}
\label{samethetas}
Suppose $\Gamma_1$ and $\Gamma_2$ form a non-isomorphic permuted pair (see Definition \ref{Cprime}). Then there exist finite-sheeted covers of $D_{\Gamma_1}$ and $D_{\Gamma_2}$ that are homotopic but not homeomorphic.   
\end{lemma}

\noindent 
\textit{Proof.} Note that $\mathcal{D}_{\Gamma_1}$ and $\mathcal{D}_{\Gamma_2}$ have the same number of generalized $\Theta$ graphs glued together, so the two-sheeted covers of their singular sets will be two isomorphic generalized $\Theta$ graphs $S_1$ and $S_2$. We will now construct two non-homeomorphic double covers of $S_1$ and $S_2$, which we will call $\widetilde{S_1}$ and $\widetilde{S_2}$. 

Since $\Gamma_1$ and $\Gamma_2$ are not isomorphic, $N > 3$, where $N$ as before denotes the number of generalized $\Theta$ graphs glued together in $\Gamma_1$ and $\Gamma_2$. Indeed, suppose that $\Gamma_1$ and $\Gamma_2$ are cycles of three generalized $\Theta$ graphs, $\Theta_1$, $\Theta_2$, and $\Theta_3$. Without loss of generality, suppose that in $\Gamma_1$, $\Theta_i$ has essential vertices $v_i$ and $v_{i + 1}$ and in $\Gamma_2$, $\Theta_1$ has essential vertices $v'_1$ and $v'_2$, $\Theta_2$ has essential vertices $v'_1$ and $v'_3$, and $\Theta_3$ has essential vertices $v'_2$ and $v'_3$. Then there is a graph isomorphism $f: \Gamma_1 \rightarrow \Gamma_2$, defined by $f(v_1) = v'_2$, $f(v_2) = v'_1$, and $f(v_3) = v'_3$ (with $f$ defined on the valence 2 vertices in the natural way). We point this out since the construction detailed below will not work for $N = 3$.

To construct $\widetilde{S_1}$, fix a cyclic ordering on a set of four vertices $\{\tilde{v}_1, \tilde{v}_2, \tilde{v}_3, \tilde{v}_4\}$. Construct edges between $\tilde{v}_1$ and $\tilde{v}_2$, and between $\tilde{v}_3$ and $\tilde{v}_4$, labeled with all integers $I$ such that $1 \leq I \leq N - 1$. Then construct one edge between $\tilde{v}_2$ and $\tilde{v}_3$, as well as $\tilde{v}_4$ and $\tilde{v}_1$ and label those edges with $N$.

According to the assumptions, for every generalized $\Theta$ graph $\Theta_i$ between vertices $v_i$ and $v_{i + 1}$ in $\Gamma_1$, there is an isomorphic generalized $\Theta$ graph $\Theta'_j$ in $\Gamma_2$ between $v'_j$ and $v'_{j + 1} \in \Gamma_2$. We can therefore assume there exists some $\Theta'_j \subset \Gamma_2$ that is isomorphic to $\Theta_1 \subset \Gamma_1$, and some $\Theta'_k \subset \Gamma_2$ that is isomorphic to $\Theta_N \subset \Gamma_1$. Without loss of generality, assume that $j < k$. Construct $\widetilde{S_2}$ with vertices $\{\tilde{v'}_1, \tilde{v'}_2, \tilde{v'}_3, \tilde{v'}_4\}$ and construct edges between $\tilde{v}_1$ and $\tilde{v'}_2$ as well as $\tilde{v'}_3$ and $\tilde{v'}_4$ labeled with all integers $I$ such that $j + 1 \leq I \leq k$. Then construct edges between $\tilde{v'}_2$ and $\tilde{v'}_3$ as well as $\tilde{v'}_4$ and $\tilde{v'}_1$ labeled with all integers $I$ such that $k + 1 \leq I \leq N$ or $1 \leq I \leq j$. 

In the Davis orbicomplexes, the isomorphic $\Theta$ graphs $\Theta_1 \subset \Gamma_1$ and $\Theta'_j \subset \Gamma_2$ give rise to identical sets of orbifolds glued to edges labeled $1$ and $2$ in $D_{\Gamma_1}$ and $j$ and $j + 1$ in $D_{\Gamma_2}$. As a result, 
if there is a cycle in $X_1$ labeled with $1$ and $2$ and a cycle of the same length in $X_2$ labeled with $j$ and $j + 1$, the sets of jester hats glued to them will be identical. The same holds for $\Theta_N \subset \Gamma_1$ and $\Theta'_k \subset \Gamma_2$. Note that in $\widetilde{S_1}$, there are two four-cycles with jester hats glued to them, one labeled with $1$ and $N$ and the other with $N$ and $N - 1$. The other cycles with jester hats glued to them are all two-cycles labeled with $I$ and $I + 1$ where $1 \leq I \leq N - 2$. On the other hand, $\widetilde{S_2}$ is a generalized $\Theta$ graph with two cycles of length four that have jester hats glued to them- one labeled with $k$ and $k + 1$ and the other with $j$ and $j + 1$. The other cycles with jester hats glued to them are labeled with $I$ and $I + 1$ where $I \neq j, k$ are two-cycles. Note that the regular neighborhood of both $\widetilde{S_1}$ and $\widetilde{S_2}$ is the $(2N - 2)$-holed sphere $S_{0, 2N - 2}$, and there is a bijective correspondence between sets of jester hats glued to boundary components of $S_{0, 2N - 2} \subset \widetilde{S_1}$ and sets of jester hats glued to boundary components of $S_{0, 2N - 2} \subset \widetilde{S_2}$. Thus, $X_1$ and $X_2$ are homotopic but not homeomorphic because $\widetilde{S_1}$ and $\widetilde{S_2}$ are not homeomorphic. For an example, refer to Figure \ref{fig:homotopic}.\\
\qed 

\begin{figure*}
    \centering
    \includegraphics[width=\textwidth]{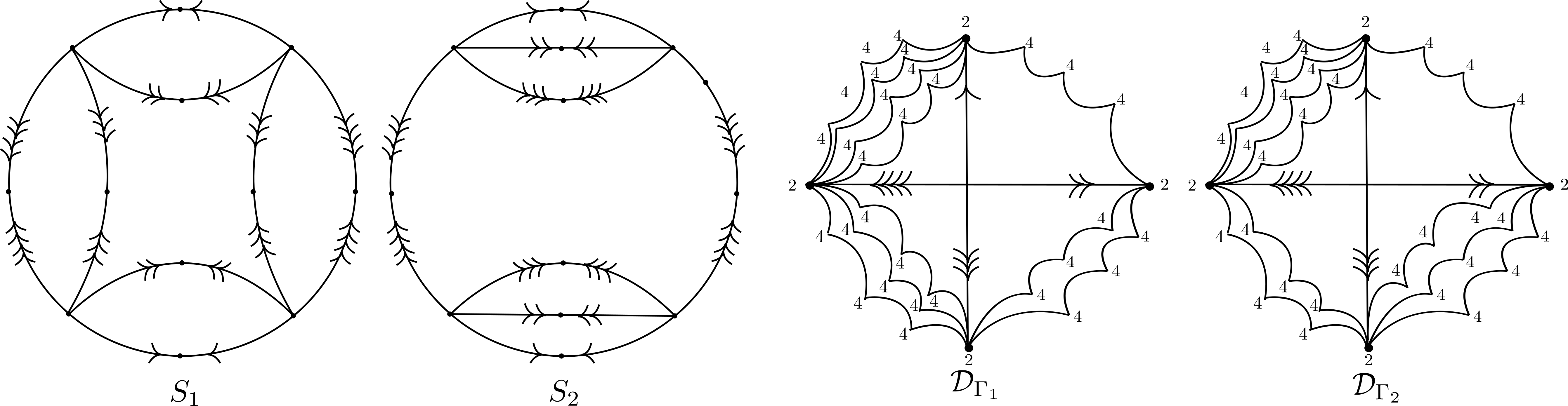}
    \caption{Here, the defining graphs $\Gamma_1$ and $\Gamma_2$ are the permuted pair from Figure \ref{fig:pp}, so they satisfy the conditions of Lemma \ref{samethetas}. The singular sets $S_i$ of $X_i$, which are four-sheeted covers of $\mathcal{D}_{\Gamma_i}$ for $i = 1, 2$ are shown on the left. Note that $X_1$ and $X_2$ are homotopic but not homeomorphic.}
    \label{fig:homotopic}
\end{figure*}

As a segue into our next section, we make the following remark. 

\begin{remark} Finding necessary and sufficient conditions for topological rigidity is a very nuanced task. In our setting, by Theorem \ref{main2}, we know a topologically rigid set cannot contain the finite-sheeted covers of $\mathcal{D}_{\Gamma_1}$ and $\mathcal{D}_{\Gamma_2}$, where $\Gamma_1, \Gamma_2$ are strongly repetitive or form a permuted pair (see Definition \ref{Cprime}). Unfortunately, 
 simply excluding finite-sheeted covers of all $\mathcal{D}_{\Gamma}$ where $\Gamma$ is repetitive and part of a permuted pair is not sufficient for constructing a topologically rigid set. For example, for $\mathcal{D}_{\Gamma_1}$ and $\mathcal{D}_{\Gamma_2}$ from Figure \ref{fig:Cdoubleprimeex}, $\Gamma_1$ and $\Gamma_2$ are not composed of the same set of generalized $\Theta$ graphs glued together. Furthermore, for both defining graphs, there do not exist pairs of commensurable Euler characteristic vectors of generalized $\Theta$ graphs in $\Gamma_1$ and $\Gamma_2$. Nevertheless, there exist eight- and four-sheeted covers of $\mathcal{D}_{\Gamma_1}$ and $\mathcal{D}_{\Gamma_2}$ respectively that are homotopic but not homeomorphic. 
\begin{figure}[h!]
    \centering
    \includegraphics[width=\textwidth]{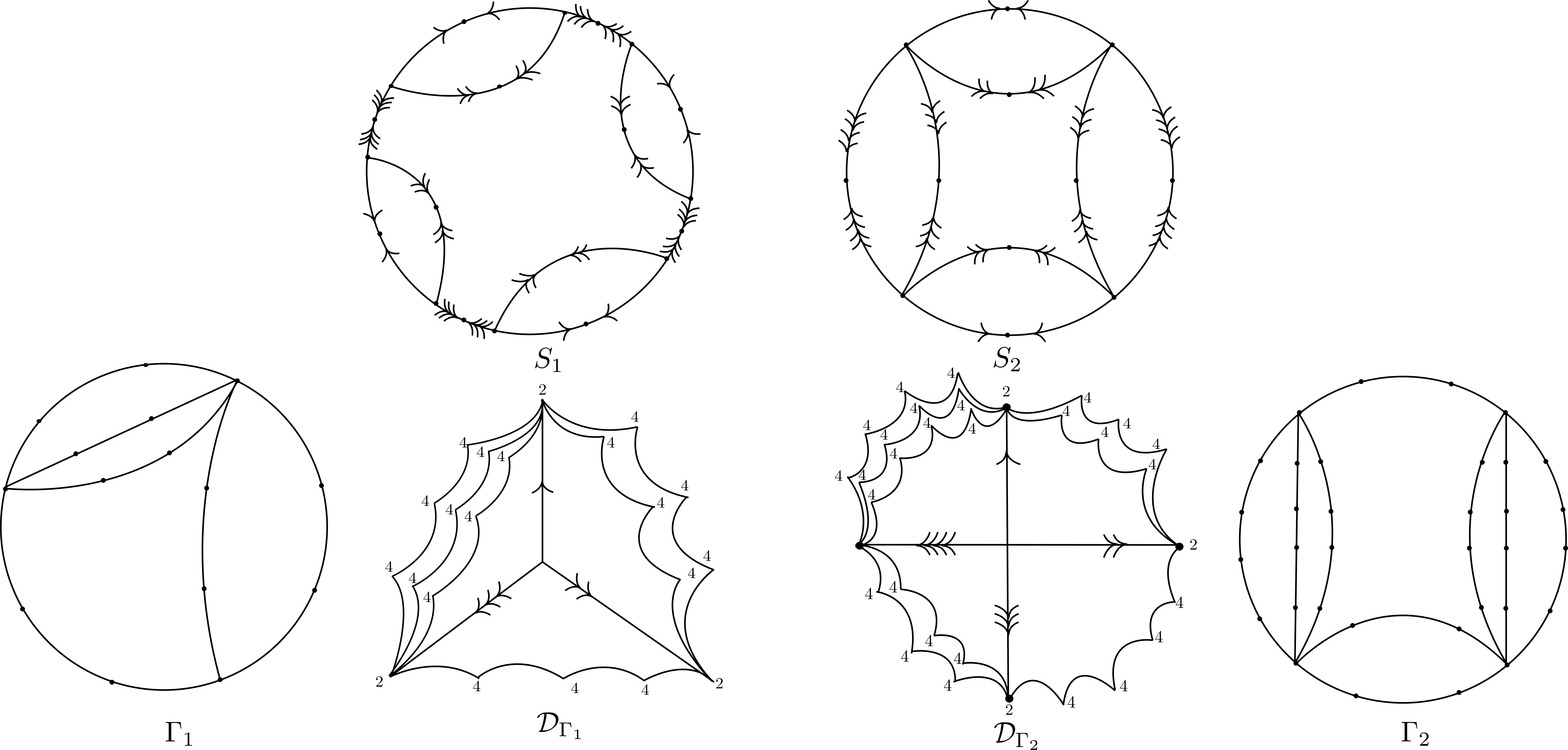}
    \caption{$S_1$ and $S_2$ are non-homeomorphic singular sets of homotopic eight- and four-sheeted covers of $\mathcal{D}_{\Gamma_1}$ and $\mathcal{D}_{\Gamma_2}$.
    Thus, neither $\Gamma_1$ nor $\Gamma_2$ is repetitive and $\Gamma_1$ and $\Gamma_2$ do not form a permuted pair, yet any set $\mathcal{X}''$ that contains all finite-sheeted covers of $\mathcal{D}_{\Gamma_1}$ and $\mathcal{D}_{\Gamma_2}$ is not topologically rigid. }
    \label{fig:Cdoubleprimeex}
\end{figure}
\end{remark} 

\section{A topologically rigid set}
\label{sec:4}

In this section, we introduce a class of finite-sheeted covers of Davis orbicomplexes that is topologically rigid. Since topological rigidity is difficult to achieve, many assumptions are necessary, so our class does not contain the complete set of finite-sheeted covers of Davis orbicomplexes. In particular, to find rigid classes, we not only need restrictions on the defining graphs but also on the singular sets of the covers.

A key tool in the proof of Theorem \ref{main} is Lafont's topological rigidity result from \cite{LaFont}. Lafont's result involves (simple, thick, 2-dimensional) \textit{hyperbolic P-manifolds} (see \cite{LaFont} Definition 2.3), which, roughly speaking, consist of compact surfaces with boundary identified along their boundaries. The gluing curves form the \textit{singular set} of the hyperbolic P-manifold. Additionally, there is the important restriction that the singular set of a hyperbolic P-manifold consists of disjoint unions of circles. We now restate the theorem we will use: 

\begin{theorem}[Lafont \cite{LaFont}, Theorem 1.2]
Let $X_1, X_2$ be a pair of simple, thick, 2-dimensional hyperbolic P-manifolds, and assume that $\phi: \pi_1(X_1) \rightarrow \pi_1(X_2)$ is an isomorphism. Then there exists a homeomorphism $\phi: X_1 \rightarrow X_2$ that induces $\phi$ on the level of the fundamental groups.
\label{Lafont} 
\end{theorem}

In order to use Theorem \ref{Lafont}, we impose a restriction on the finite covers of Davis orbicomplexes we are examining, which we state below. 

\begin{assumption}
$X$ is homotopic to an orbicomplex $Y$ that consists of jester hats glued along their boundaries to the boundaries of an $h$-holed genus $g$ surface $S_{g, h}$. Furthermore, in $Y$, each boundary of $S_{g, h}$ has at least one jester hat glued to it.

\label{ass1}
\end{assumption}

Recall that a graph $\Gamma$ is said to be \emph{3-convex} if every edge between its essential vertices has at least 3 subdivisions. Note that if a cycle of generalized $\Theta$ graphs $\Gamma$ is 3-convex, then $W_{\Gamma}$ is hyperbolic since $\Gamma$ is square-free. Although the converse is not true, we impose the 3-convexity condition in our proof to ensure our construction of hyperbolic P-manifold lifts of Davis orbicomplexes will work. 

We now use Assumption \ref{ass1} to prove a lemma that will be important in the proof of our main result of the section. 
\begin{lemma}
\label{maps}
Let $X_1$ and $X_2$ be finite covers of Davis orbicomplexes $D_{\Gamma_1}$ and $D_{\Gamma_2}$, where $\Gamma_1$ and $\Gamma_2$ are 3-convex and satisfy Assumption \ref{ass1}, and suppose $\pi_1(X_1) \cong \pi_1(X_2)$. Then the isomorphism $f: \pi_1(X_1) \rightarrow \pi_1(X_2)$ induces a bijection $f_{\ast}$ between jester hats of $X_1$ and $X_2$ and for a jester hat $\mathcal{O}_1 \subset X_1$, if $f_{\ast}(\mathcal{O}_1) = \mathcal{O}_2 \subset X_2$, then $\mathcal{O}_1$ and $\mathcal{O}_2$ are homeomorphic. Furthermore, if $S_1$ and $S_2$ are singular subsets of $X_1$ and $X_2$ respectively, if $\gamma_1 \subset S_1$ is the boundary component of $\mathcal{O}_1$, then $f_{\ast}(\gamma_1) = \gamma_2$ where $\gamma_2$ is the boundary component of $\mathcal{O}_2$. 
\end{lemma} 

\begin{proof} Suppose $X_1$ and $X_2$ are two orbicomplexes with 3-convex defining graphs where $\pi_1(X_1) \cong \pi_1(X_2)$. We first use the construction from Proposition 3.2 of \cite{Stark}: let $X_i$ be a finite cover of $\mathcal{D}_{{\Gamma}_i}$ for $i = 1, 2$. Each jester hat in $X_i$ with $p$ cone points lifts to a orbifold with $2(p - 2)$ cone points and two boundary components, which in turn has a two-sheeted cover $S_{g, 4}$, where $g = \frac{2(p - 2)}{2} - 1$. Then, we glue each boundary component of $S_{g, 4}$ to a copy of $S_i$, the singular set of $X_i$, to obtain a torsion-free four-sheeted cover. We will call these torsion-free covers $\widehat{X_1}$ and $\widehat{X_2}$. See Figure \ref{fig:tower} for an illustration of the construction.

Since abstract commensurability is an equivalence relation, $\pi_1(\widehat{X_1})$ and $\pi_1(\widehat{X_2})$ are abstractly commensurable, so there exist finite-sheeted covers $\mathcal{Y}_1$ and $\mathcal{Y}_2$ such that $$\pi_1(\widehat{X_1}) \geq \pi_1(\mathcal{Y}_1) \cong \pi_1(\mathcal{Y}_2) \leq \pi_1(\widehat{X_2}).$$ Note that $\widehat{X_1}$ and $\widehat{X_2}$ are homotopic to hyperbolic P-manifolds; take the regular neighborhoods of their singular sets. As a result, finite-sheeted covers of $\widehat{X_1}$ and $\widehat{X_2}$ are also homotopic to hyperbolic P-manifolds by Nielsen-Schreier, so it follows that $\mathcal{Y}_1$ and $\mathcal{Y}_2$ are homotopic to hyperbolic P-manifolds $Y_1$ and $Y_2$. By Corollary 3.5 of \cite{LaFont}, a homotopy $\phi$ between hyperbolic P-manifolds induces a bijection between homeomorphic chambers (hyperbolic manifolds with boundary) of $Y_1$ and $Y_2$. Additionally, for a chamber $C_1 \subset Y_1$ with boundary component $\gamma_1$, if $\phi(C_1) = C_2 \subset Y_2$, then $\phi(\gamma_1) = \gamma_2$ where $\gamma_i$ is a boundary component of $C_i$ for $i = 1, 2$. 
Using these results, we can conclude that surfaces with boundary in $\mathcal{Y}_1$ are mapped bijectively, and homeomorphically, to surfaces in $\mathcal{Y}_2$, as the maps are preserved under homotopy. The homotopy lifting property then gives us the statement of the lemma. 

Alternatively, observe that the Davis complex $\Sigma_{\Gamma}$, the universal cover of $\mathcal{D}_{\Gamma}$, is CAT(0) and thus contractible. Since $W_{\Gamma}$ acts freely on $\Sigma_{\Gamma}$, it follows that $\mathcal{D}_{\Gamma}$ is a classifying space of $W_{\Gamma}$, or equivalently a $K(W_{\Gamma}, 1)$ space. Furthermore, finite covers of $\mathcal{D}_{\Gamma}$ are quotients of $\Sigma_{\Gamma}$ by a free action as well, so $X_1$ and $X_2$ are classifying spaces for the same finite-index subgroup of $W_{\Gamma}$. As a result, since $\pi_1(X_1) \cong \pi_1(X_2)$, $X_1$ and $X_2$ are homotopy equivalent by Whitehead's Theorem. This allows us to construct a shorter tower of covers since a homotopy between $X_1$ and $X_2$ induces a homotopy between $\widehat{X_1}$ and $\widehat{X_2}$, which are homotopic to hyperbolic P-manifolds. We can then apply Theorem \ref{Lafont} to obtain our result.  
\qed 
\end{proof}

\begin{figure*}
    \centering
    \includegraphics[width=0.95\textwidth]{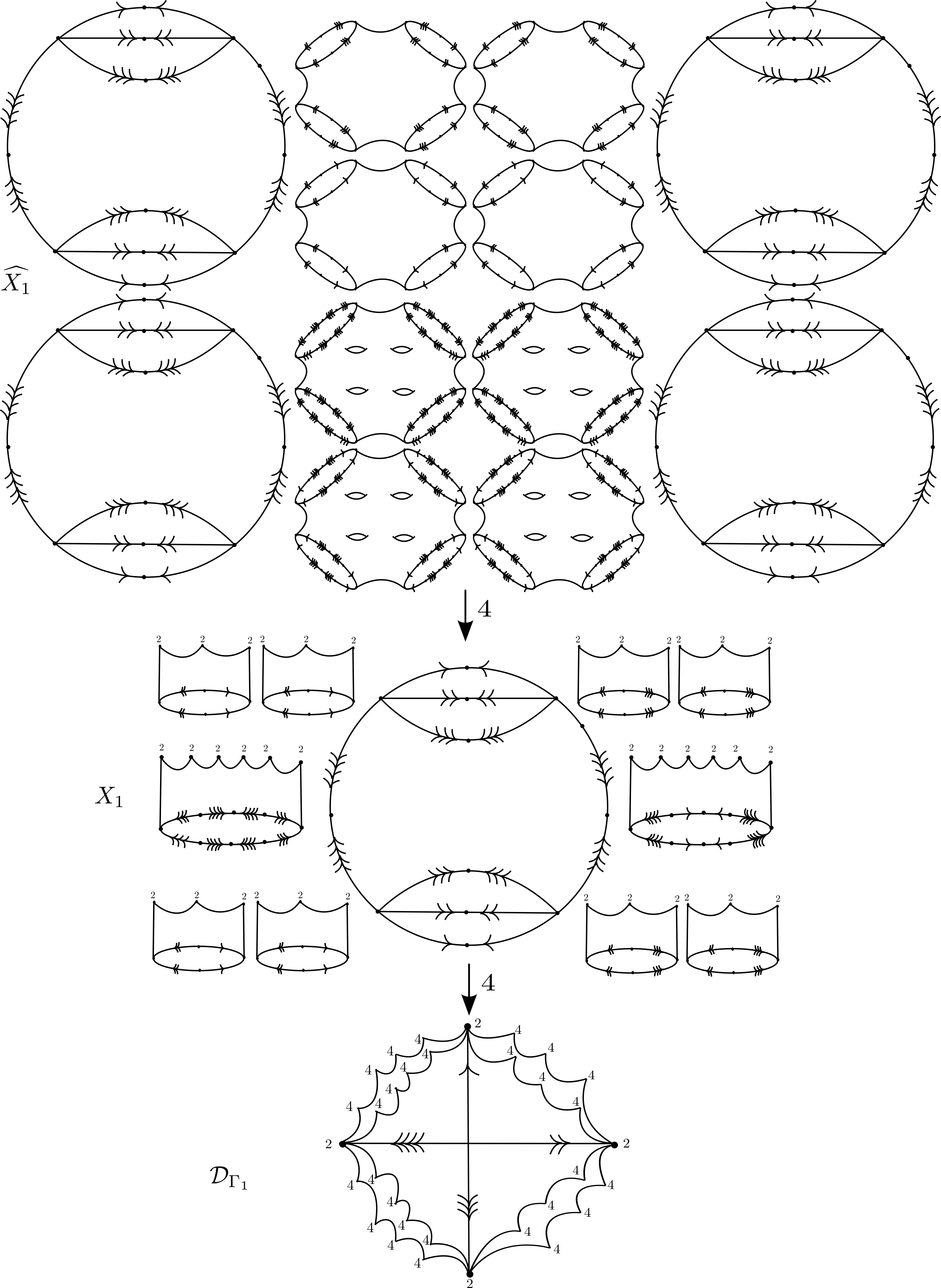}
    \caption{A tower of covers constructed in the proof of Lemma \ref{maps}. Note that $\widehat{X_1}$ is homotopic to a hyperbolic P-manifold.}
    \label{fig:tower}
\end{figure*}

We stress that Assumption \ref{ass1} is key for the proof of Lemma \ref{maps}. For example, let $S_{g, n}$ be a genus $g$ surface with $n$ boundary components. Recall the \textit{graph genus} of a graph $G$ is the minimal genus of an orientable surface into which $G$ can be embedded. In general, a cover of a Davis orbicomplex $X$ is homotopic to an orbicomplex consisting of jester hats identified along their boundaries to a set of simple closed curves $C$ on $S_{g, n}$ since every graph has a genus (see \cite{white}). Analyzing $X$ can be difficult since there is no guarantee of the four-sheeted hyperbolic P-manifold cover constructed in Lemma \ref{maps}, as the lifts of $C$ may not be a disjoint union of circles. For example, consider Figure \ref{fig:K33}, which depicts a six-sheeted cover of a Davis orbicomplex $\mathcal{D}_{\Gamma}$. The singular set of $X$ is the complete bipartite graph $K_{3, 3}$, which is homotopic to $S_{1, 3}$. Since the jester hats are not identified along disjoint circles, Lafont's rigidity result is not available for use and the proof for Lemma \ref{maps} does not work.

\begin{figure*}[h!]
    \centering
    \includegraphics[width=\textwidth]{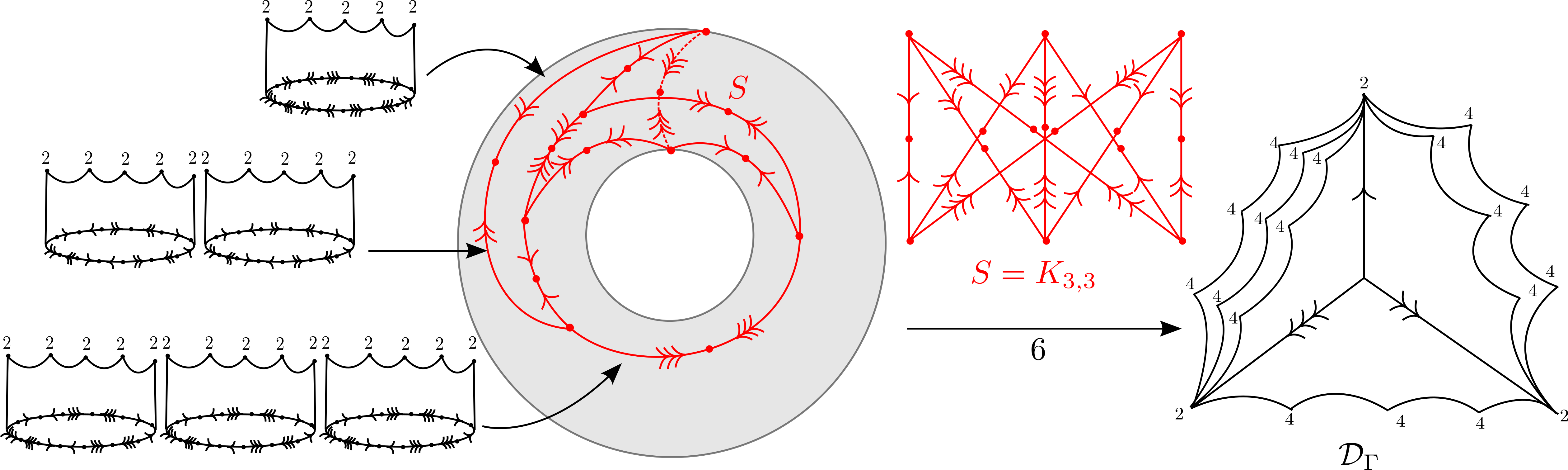}
    \caption{An example of a finite-sheeted cover of a Davis orbicomplex $\mathcal{D}_{\Gamma}$ that does not satisfy Assumption \ref{ass1}. The singular set $S$ is not necessarily planar; in this example, $S$ (depicted in red) embeds on a torus.}
    \label{fig:K33}
\end{figure*}

In order to determine whether two finite-sheeted covers of Davis orbicomplexes are homeomorphic, which we need to do to determine topological rigidity, we need to check that their singular sets are homeomorphic. Unfortunately, since finite covers of the singular sets are graphs, determining topological rigidity therefore requires solving a graph isomorphism problem, which has a high computational complexity. Recall that a \textit{complete graph invariant} is combinatorial tool for determining whether a pair of graphs in a family of graphs is isomorphic. To simplify our problem, we will define a family of singular sets $\mathcal{S}$ of finite-sheeted covers of a Davis complex $\mathcal{D}_{\Gamma}$ with an easily computable complete graph invariant, which we now introduce.   

\begin{definition}[Cycle count vectors]
\label{vecs} 
Consider $X$, a $2d$-sheeted cover of a Davis orbicomplex, where $d > 0$ is any arbitrary integer. Recall that for a cycle of generalized $\Theta$ graphs $\Gamma$, its associated Davis orbicomplex $\mathcal{D}_{\Gamma}$ has a singular set that is a star graph with $N$ edges, which we will label with integers $i = 1, 2, ..., N$. If an edge $e = [v, w]$ is labeled by $i$, we will write $e = [v, w]_i$. Fix a cyclic labeling of the edges, which will lift to a labeling in any cover of the singular set. For $1 \leq i \leq N$, let $x_i = (x_{i, 1}, x_{i, 2}, ..., x_{i, d})$ be a vector where $x_{i, j}$ is the number of cycles in $S$ of length $2j$ that are labeled with $i$ and $i + 1$. Recall that as usual, we are counting the edges between two essential vertices as one edge. Note that $x_i$ is a vector of length $d$ since the possible cycle lengths of a $2d$-sheeted cover will range from $2$ to $2d$. 
For an example, refer to Figure \ref{fig:homotopic}. Let the labeling of an edge of $S_i$ be the number of arrows seen on the edge, so $N = 3$. For $S_1$, the set of cycle count vectors is $x_1 = (2, 0)$, $x_2 = (0, 1)$, $x_3 = (2, 0)$, and $x_4 = (0, 1)$ since there are two 2-cycles labeled with $1$ and $2$, two 2-cycles labeled with $3$ and $4$, one 4-cycle labeled with $2$ and $3$, and one 4-cycle labeled with $4$ and $1$. Similarly, for $S_2$, the set of cycle count vectors is $x_1 = (2, 0)$, $x_2 = (2, 0)$, $x_3 = (0, 1)$ and $x_4 = (0, 1)$. As we will see soon, since the cycle count vectors are different, $S_1$ and $S_2$ are not isomorphic.  
\end{definition}

We now define a family of singular sets of finite-sheeted covers of a Davis complex $\mathcal{D}_{\Gamma}$. Recall that the double cover of the singular set of $\mathcal{D}_{\Gamma}$ is itself a generalized $\Theta$ graph $\Theta_N$ with $N$ branches, where $N$ is the number of generalized $\Theta$ graphs in the defining graph $\Gamma$ (see Construction \ref{double}). A double cover of $\Theta_N$ is a cycle of four (possibly trivial) generalized $\Theta$ graphs $S'$ with valence $N$ vertices. Note there exists some $a, b \in \mathbb{Z}_{\geq 0}$ and $a + b = N$ such that adjacent essential vertices of $S'$ either have $a$ or $b$ branches between them. For example, in Figure \ref{fig:homotopic}, all the essential vertices in $S_1$ and $S_2$ have valence $N = 4$ since the original defining graph consisted of four generalized $\Theta$ graphs glued together. In $S_1$, $a = 2$ and $b = 2$, and in $S_2$, $a = 1$ and $b = 3$. 

\begin{construction} [A special class of singular sets $\mathcal{S}$] To begin our construction, take $S'$, a four-sheeted cover of the singular set of $\mathcal{D}_{\Gamma}$, which is a cycle of four generalized $\Theta$ graphs each with either $a$ or $b$ branches. Then arbitrarily choose two adjacent vertices of valence greater than two, $v_i$ and $v_{i + 1}$ with $n_i = a$ or $b$ edges between them. Delete some fixed number of (subdivided) edges $j$ between them, where $1 \leq j \leq n_i$, add two essential vertices $u_i$ and $u_{i + 1}$, and add $j$ edges between $v_i$ and $u_i$ as well as $v_{i + 1}$ and $u_{i + 1}$. Finally, add $N - j$ edges between $u_i$ and $u_{i + 1}$. Then arbitrarily choose two other adjacent essential vertices and repeat the process any finite number of times. See Figure \ref{fig:Sexample} for an example of an element of $\mathcal{S}$; at each step, two edges are added between the new essential vertices. 

\begin{figure}[h]
    \centering
    \includegraphics[width=\textwidth]{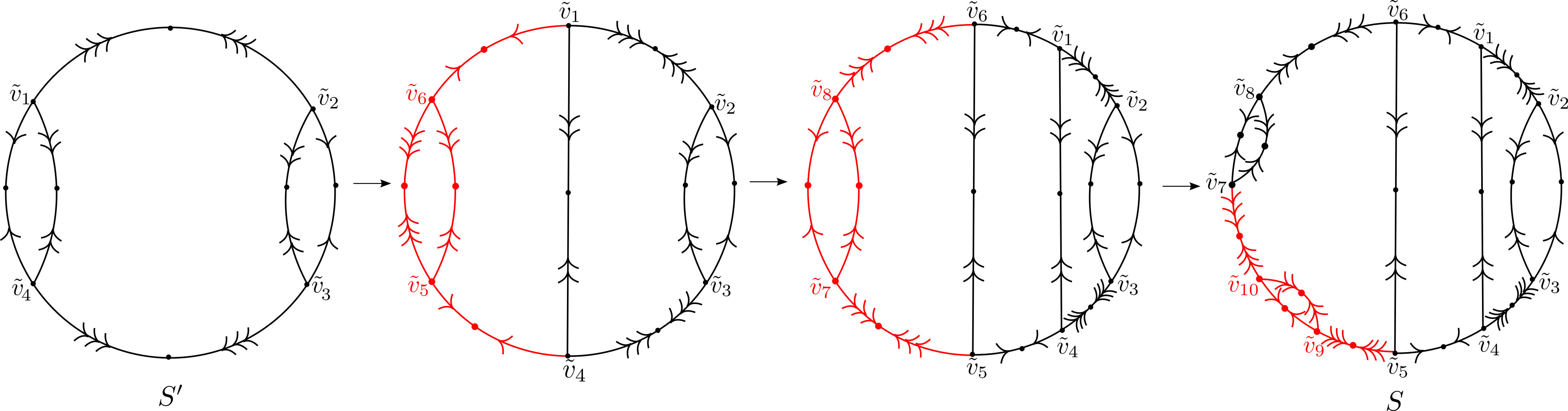}
    \caption{Construction of a $S \in \mathcal{S}$ from a cycle of generalized $\Theta$ graphs. The edges and vertices added at each step are depicted in red.}
    \label{fig:Sexample}
\end{figure}
\label{singularset}
\end{construction}

Let $S \in \mathcal{S}$ be a graph that can be constructed from the process described above. By construction, $S$ covers the original singular set of the Davis orbicomplex. Notice $\mathcal{S}$ describes singular sets, not defining graphs, so the graphs in $\mathcal{S}$ are not restricted to cycles of generalized $\Theta$ graphs. We now introduce a second assumption: 

\begin{assumption}
    \label{ass2} 
    The singular set $S$ of $X$ is an element of the set $\mathcal{S}$ defined in Construction \ref{singularset}.
\end{assumption}

We are now ready to introduce a topologically rigid class finite covers of Davis orbicomplexes.

\begin{theorem} 
\label{main} Suppose $\Gamma$ is 3-convex and not repetitive (see Definition \ref{C}). Let $\mathcal{X}''$ contain all finite-sheeted covers of $\mathcal{D}_{\Gamma}$ that satisfy Assumptions \ref{ass1} and \ref{ass2}. Then $\mathcal{X}''$ is topologically rigid.  
\end{theorem} 

We now give an outline of the proof of Theorem \ref{main} (for the actual proof, see the end of the section). We first show that the statement of the theorem reduces to a graph isomorphism problem on singular sets of two finite covers of Davis orbicomplexes (see Lemma \ref{cor}). We then show that if two finite covers of Davis orbicomplexes satisfying the conditions listed in Theorem \ref{main} are homotopic, then they have the same cycle count vectors. Finally, in Lemma \ref{b}, we show that cycle count vectors are a complete graph invariant for singular sets in $\mathcal{S}$ detailed in Construction \ref{singularset}.   

 \begin{remark} 
 Recall that if $X_1$ and $X_2$ are homeomorphic, then $W_{\Gamma_1}$ and $W_{\Gamma_2}$ are commensurable. Given that Assumption \ref{ass1} is true, the converse is very much false. Figure 1.2 of \cite{DST} gives some examples of pairs of defining graphs $\{\Gamma_i, \Gamma'_i\}$ ($i = 1, 2, 3$) of commensurable RACGs, which we will now reference. We can check that none of the finite-sheeted covers of $\mathcal{D}_{\Gamma_i}$ and $\mathcal{D}_{\Gamma_i'}$ can be homotopic by Lemma \ref{maps}, even though the RACGs $W_{\Gamma_i}$ and $W_{\Gamma_i'}$ are commensurable. To see the full commensurability classification for cycles of generalized $\Theta$ graphs, refer to Theorem 1.12 of \cite{DST}. 
\end{remark} 

\begin{definition} We say that a graph homeomorphism $\bar{f}: G_1 \rightarrow G_2$ is \textit{label-preserving} if for all $[v, w]_i \in G_1$, $\bar{f}([v, w]_i) \subset G_2$ is also labeled by $i$. \end{definition} 

We now prove the following useful result involving singular sets of finite covers of Davis orbicomplexes.

\begin{lemma}
Let $X_1$ and $X_2$ satisfy the assumptions from Lemma \ref{maps}. Then every label-preserving graph homeomorphism $\bar{f}: S_1 \rightarrow S_2$ between the singular subsets of $X_1$ and $X_2$ will induce a homeomorphism $f: X_1 \rightarrow X_2$.
\label{cor}
\end{lemma} 

\begin{proof} Suppose the singular sets $S_1 \subset X_1$ and $S_2 \subset X_2$ are homeomorphic. Then the vertices $v_j \in V(S_1)$ with valence greater than two map bijectively to the vertices $v'_j \in V(S_2)$ with the same valence, and if $\bar{f}(v_1) = v_1'$ and $\bar{f}(v_2) = v_2'$, then for every edge $[v_1, v_2]_i$ labeled with $i$, $\bar{f}([v_1, v_2]_i) = [v_1', v_2']_i$. As a result, every cycle $\gamma_1 \subset S$ labeled with $i$ and $i + 1$ ($\mod N$) is bijectively mapped to a cycle of the same length in $\gamma_2 \subset S_2$ labeled with $i$ and $i + 1$ ($\mod N$). By Lemma \ref{maps}, for every jester hat $\mathcal{O}_1$ with $c$ cone points glued to $\gamma_1$, there must also be a jester hat $\mathcal{O}_2$ with $c$ cone points glued to $\gamma_2 = \bar{f}(\gamma_1)$. As a result, $\bar{f}$ induces a homeomorphism $f: X_1 \rightarrow X_2$ where $f(S_1) = S_2$ and $f(\mathcal{O}_1) = \mathcal{O}_2$.
\qed 
\end{proof} 

Recall that a graph $G$ with genus $g$ can be embedded into a genus $g$ surface $S_g$. The edges of the graph will divide $S_g$ into regions called \textit{faces}. Let $\lvert V\rvert$ denote the number of vertices of $G$, $\lvert E \rvert$ the number of edges, and $\lvert F \rvert$ the number of faces. For planar graphs, we can calculate the number of faces using Euler's formula, $2 = \lvert V \rvert - \lvert E \rvert + \lvert F \rvert$. In general, using the definition of Euler characteristic for simplicial complexes, for a graph with genus $g$, $2 - 2g = \lvert V \rvert - \lvert E \rvert + \lvert F \rvert$. 

\begin{lemma}
\label{a}
Let $\mathcal{X}''$ be the set of finite-sheeted covers of a single Davis orbicomplex $\mathcal{D}_{\Gamma}$, where $\Gamma$ is 3-convex and not repetitive. Suppose $X_1, X_2 \in \mathcal{X}''$ are homotopic, and their singular sets $S_1$ and $S_2$ satisfy the conditions listed in Theorem \ref{main}. Then $x_i = x'_i$ for all $1 \leq i \leq N$, where $x_i$ and $x_i'$ are the cycle count vectors of $X_1$ and $X_2$ respectively.   
\end{lemma} 

\begin{proof} 
\begin{comment}
As a consequence of Lemma \ref{maps}, two orbicomplexes can only have isomorphic fundamental group if they have the same sets of jester hats. Furthermore, if there are $k$ jester hats glued to a cycle in $X_1$, there must also be a set of $k$ jester hats glued to a cycle in $X_2$, and the two collections of jester hats must be the same. Thus, if $X_1$ and $X_2$ are covers of the same Davis orbicomplex, they cannot have isomorphic fundamental groups if the degrees of the covers are different, as the collection of jester hats will be different. 
\end{comment}
First, we show that if $\Gamma$ consists of $N$ essential vertices, and $X_i$ are $d_i$ sheeted covers of $\mathcal{D}_{\Gamma}$ for $i = 1, 2$, then $d_1 = d_2$ necessarily. As usual, we will denote $S_i$ to be the singular set of $X_i$. Since $X_1$ and $X_2$ are homotopic, by Lemma 4.3, they consist of the same sets of jester hats identified along their boundary components to some $F$-holed genus $g$ surface $S_{g, F}$. Note that $F$ is also the number of faces of both $S_1 \subset X_1$ and $S_2 \subset X_2$. Note the number of vertices in $S_i$ is $d_i$ and the number of edges is $\frac{d_iN}{2}$ for $i = 1, 2$, so using the definition of Euler characteristic, we have: 
$$d_1 - \frac{d_1N}{2} + F = 2 - 2g = d_2 - \frac{d_2N}{2} + F \implies d_1\big(2 - N\big) = d_2\big(2 - N\big).$$ Thus, $d_1 = d_2$ necessarily. Then $X_1$ and $X_2$ are finite sheeted covers of the same degree of the same Davis orbicomplex $\mathcal{D}_{\Gamma}$.

%Since $\Gamma$ is not repetitive, there do not exist $\Theta = \Theta(m_1, m_2, ..., m_k), \Theta' = \Theta(m'_1, m'_2,..., m'_k) \subset \Gamma$ and $K, L$ such that $\frac{K}{2}(m_i - 3) + 2 = \frac{L}{2}(m'_i - 3) + 2$ for all $m_i$ and $m'_i$. Thus, there do not exist $K$ and $L$ such that $\frac{K}{2}(r_i - 3) + 2 \not\equiv \frac{L}{2}(r'_i - 3) + 2$

Let $\Theta = \Theta(n_1, n_2, ..., n_k)$ and $\Theta' = \Theta(n'_1, n'_2,...,n'_{k'})$ be two arbitrary $\Theta$ graphs in $\Gamma$. Recall that by Lemma \ref{jh}, the number of cone points $c_i$ of a jester hat corresponding to $b_i$, the $i$th branch of $\Theta$, is $\frac{d}{2}(r_i - 3) + 2$ where $d$ is the index of the cover the jester hat corresponds to. Similarly, the number of cone points $c'_j$ of a jester hat corresponding to $b'_j$, the $j$th branch of $\Theta'$, is $\frac{d}{2}(r'_j - 3) + 2$. Note that $\Gamma$ is not repetitive, so there do not exist $K, L$ such that $K(\frac{1 - n_1}{4}, \frac{1 - n_2}{4}, ..., \frac{1 - n_k}{4}) = L(\frac{1 - n'_1}{4}, \frac{1 - n'_2}{4}, ..., \frac{1 - n'_{k'}}{4})$. So in particular, $(n_1 - 1, n_2 - 1, ..., n_k - 1) \neq (n'_1 - 1, n'
_2 - 1, ..., n'_{k'} - 1)$ and since $n_i$ and $n'_j$ are equal to $r_i - 2$ and $r'_j - 2$ respectively, it follows that $\{\frac{d}{2}(r_i - 3) + 2\}_{1 \leq i \leq k} \neq \{\frac{d}{2}(r'_j - 3) + 2\}_{1 \leq j \leq k'}$. Thus, in any finite-sheeted cover of $D_{\Gamma}$, there are different sets of jester hats glued to cycles with different labels since the sets of cone point counts are different. In particular, jester hats in $X_1$ that are lifts of orbifolds corresponding to a $\Theta$ graph in $\Gamma$ must map to a collection of jester hats in $X_2$ that are lifts of orbifolds corresponding to the same $\Theta$ graph. In order for the sets of jester hats to be the same, cycles of length $2j$ labelled with $i$ and $i + 1$ must map to cycles of the same length and also labeled with $i$ and $i + 1$. As a result, $x_i = x'_i$. 

%, which means $\frac{K}{2}(r_i - 3) + 2 \not\equiv \frac{L}{2}(r'_i - 3) + 2$ where as before for some $r_i$ and $r'_1$, $r_i$ and $r'_i$ are the number of reflection edges of the orbifolds corresponding to the $i$th branches of $\Theta$ and $\Theta'$ respectively. 
\qed 
\end{proof} 

\begin{lemma}
\label{b}
Let $\mathcal{X}''$ be the class of finite-sheeted covers of Davis orbicomplexes from Theorem \ref{main}.  Suppose $S_1, S_2 \in \mathcal{S}$ are the singular sets of two $2d$- sheeted covers $X_1, X_2 \in \mathcal{X}''$. If $x_i = x'_i$ for all $1 \leq i \leq N$, then there exists a label-preserving homeomorphism $\bar{f}: S_1 \rightarrow S_2$. 
\end{lemma}

\begin{proof} We use induction on the degree of the covers, $2d$. Note that since the cycle count vectors are the same, the covers must be of the same degree. For the base case, suppose $d = 2$, so $2d = 4$. By Assumption \ref{ass1}, each cycle in both $S_1$ and $S_2$ must be attached to at least one jester hat. Thus, cycles of odd length are not allowed since a jester hat cannot be attached to such a cycle. Thus, the only possible cycle lengths of $S_1$ and $S_2$ are 2 and 4. By construction, every edge in $S_1$ must be attached to part of the boundary of at least one jester hat, so every edge in $S_1$ must be included into at least one cycle. In order for $S_1$ to be a connected graph with four vertices satisfying the property that every edge is part of a at least one cycle, there must be at least one four-cycle in $S_1$. The same holds for $S_2$. Since three-cycles are not allowed and each edge in $S_1$ will be included in a cycle, the only possible edges in $S_1$ are of the form $[v_i, v_{i + 1}]$, where as before, $i$ is taken modulo 4. The same holds for $S_2$. Therefore, the only possible $S_1$ and $S_2$ are graphs with $k$ edges between $v_1$ and $v_2$ as well as $v_3$ and $v_4$ (where $1 \leq k < N$), and $N - k$ edges between $v_2$ and $v_3$ as well as $v_4$ and $v_1$ (namely, a cycle of $4$ subdivided $\Theta$ graphs- see $S_1$ and $S_2$ in Figure \ref{fig:Cdoubleprimeex} for examples of such graphs). Note that such graphs will have two four-cycles and $(2N - 4)$ two-cycles. As a result, the only 4-sheeted covers that satisfy Assumption \ref{ass1} of Theorem \ref{main} consist of two sets of jester hats glued to four-cycles, and the rest of the sets of jester hats glued to two-cycles.

If $x_i = x'_i$ for all $1 \leq i \leq N$, we know that for both $S_1$ and $S_2$, there is one four-cycle labeled with $j$ and $j + 1$ for some $1 \leq j \leq N$, and another four-cycle labeled with $k$ and $k + 1$ for $k \neq j$. Without loss of generality, suppose $j < k$. Note that in $S_1$, if an edge labeled $j + 1$ is between $v_i$ and $v_{i + 1}$, then an edge labeled $k$ must necessarily also be between $v_i$ and $v_{i + 1}$. Otherwise, the edge labeled with $k + 1$ must be between $v_i$ and $v_{i + 1}$, so an edge labeled with $k$ is between $v_i$ and $v_{i - 1}$. In this case, the labels on the edges between $v_i$ and $v_{i + 1}$ will range from $j + 1$ to $k + 1$, so one of the edges must be labeled with $k$. However, there is already a $k$ edge between $v_i$ and $v_{i - 1}$, which is impossible. The same argument can be used for $S_2$ if we replace $v_i$ with $v'_i$. We can thus see that in $S_1$, for all $1 \leq i \leq N$, if edges labeled with all integers between $j + 1$ and $k$ connect $v_i$ and $v_{i + 1}$ ($v'_i$ and $v'
_{i + 1}$ in $S_2$), then the edges connecting $v_i$ and $v_{i - 1}$ ($v'_i$ and $v'_{i - 1}$ in $S_2$) are labeled with integers $1 \leq l \leq N$ such that $l \leq j$ or $l \geq k + 1$. We can then construct a homeomorphism $\bar{f}: S_1 \rightarrow S_2$ where $f(v_i) = v'_i$ for all $1 \leq i \leq N$. We can easily check that edge labels and vertex adjacencies are preserved under $\bar{f}$, completing the base case. 

Suppose the lemma holds for $d$-sheeted covers. By Assumption \ref{ass2}, there exist $\{v_i$, $v_{i + 1}\} \in V(S_1)$ and $\{v'_i$, $v'_{i + 1}\} \in V(S_2)$ with the same number of edges and the same set of labels between them. Additionally, the only other edges attached to $v_i$ and $v_{i + 1} \in V(S_1)$ are also attached to $v_{i - 1}$ and $v_{i + 2}$ respectively; the same holds for $v'_i, v'_{i + 1} \in V(S_2)$.
Note that for arbitrary $d > 0$, $S_1$ and $S_2$ must both have $(2d + 2)$-cycles for the same reason the four-sheeted covers in the base case necessarily have $4$-cycles: each edge of $S_1$ and $S_2$ is necessarily attached to a jester hat, and thus by Assumption \ref{ass1} necessarily belongs to a cycle. Suppose there are no cycles in $S_1$ and $S_2$ of length $2d + 2$. Then $S_1$ and $S_2$ would be disconnected since they are graphs with $2d + 2$ vertices. Thus, there must be at least one $(2d + 2)$-cycle in $S_1$ and $S_2$, which we will label with $m$ and $m + 1$. Note that if $v_i$ and $v_{i + 1}$ have $j$ edges between them, then the other two pairs of vertices $\{v_i, v_{i - 1}\}$ and $\{v_{i + 1}, v_{i + 2}\}$ must also necessarily have $N - j$ edges between them, and the edges between the two pairs of vertices have the same set of labels. Delete $v_i$ and $v_{i + 1}$ and the edges they are adjacent to, and construct the $j$ deleted edges between $v_{i + 2}$ and $v_{i - 1}$ to create the singular set $T_1$ of a $2d$-sheeted cover of $\mathcal{D}_{\Gamma_1}$. Construct a singular set $T_2$ of $2d$-sheeted cover of $\mathcal{D}_{\Gamma_2}$ in the same way. Let $y_i$ and $y'_i$ be the new set of cycle count vectors. Note that in total, for both $2d$-sheeted covers, we have deleted and added the same set of cycles with the same set of edge labels, so $y_i = y'_i$ for all $1 \leq i \leq N$. Then by the inductive hypothesis, there exists a homeomorphism $\bar{g}: T_1 \rightarrow T_2$. 

We then can extend $\bar{g}$ to $\bar{f}: S_1 \rightarrow S_2$. Suppose $\bar{g}(v_{i + 2}) = v'_k$ for some $v'_k \in T_2$ (note that $k$ is not necessarily equal to $i$). Then $\bar{g}(v_{i - 1})$ maps to an adjacent vertex $v_l$ such that there are $j$ edges between $v'_k$ and $v'_l$ with the same labelings as the edges between $v_i$ and $v_{i + 1}$ in the original $S_1$. Construct two vertices $u_{k}$ and $u_{l}$ in $T_2$ and $v_{i - 1}$ and $v_{i + 2}$ in $T_1$, and delete the $N - j$ edges between $v_k, v_l \in V(T_2)$ and $v_{i - 1}, v_{i + 2} \in V(T_1)$ that have the same labels as the edges added to construct $T_i$ from $S_i$ for $i = 1, 2$. Then reconstruct the $j$ deleted edges between $u_k, u_l \in V(T_2)$ and $v_{i - 1}, v_{i + 2} \in V(T_1)$ as well as $N - j$ edges $[v'_k, u_k]$ and $[v'_l, u_l]$ in $T_2$ and $[v_{i - 1}, v_i]$ and $[v_{i + 1}, v_{i + 2}]$ in $T_1$. Call the new graphs $U_1$ and $U_2$, but note that $U_1$ is identical to $S_1$. Additionally, $U_2$ is homeomorphic to $S_2$ since they are the same graph up to a relabeling of vertices. Let $\bar{g}(v) = \bar{f}(v)$ for all $v \in V(S_1)$, $\bar{f}(v_i) = u_k$ and $\bar{f}(v_{i - 1}) = u_l$, which also determines the maps between the newly added edges, giving us a label-preserving homeomorphism $\bar{f}: S_1 \rightarrow S_2$.
\qed 
\end{proof}

We now have all the tools to prove Theorem \ref{main}. 

\begin{proof}[Theorem \ref{main}] It suffices to show that for $X_1, X_2 \in \mathcal{X}$, $\pi_1(X_1) \cong \pi_1(X_2)$ implies $X_1 \overset{\text{homeo}}{\cong} X_2$. As a result of Lemma \ref{cor}, in order to show $X_1 \overset{\text{homeo}} \cong X_2$, it suffices to show there exists a label-preserving graph homeomorphism between the singular sets of $X_1$ and $X_2$ respectively. Since subdivisions of a graph  belong to the same homeomorphism class, we can delete the valence two vertices from the singular sets to obtain $S_1$ and $S_2$, and compare the graphs. By Lemma \ref{a}, $x_i = x'_i$ for $1 \leq  i \leq N$, where $x_i$ and $x'_i$ are the cycle count vectors of $S_1$ and $S_2$ defined in Definition \ref{vecs}. Then by Lemma \ref{b}, we can conclude $S_1 \overset{\text{homeo}}{\cong} S_2$, as desired. 
\qed 
\end{proof}

%Text with citations \cite{RefB} and \cite{RefJ}.
%\subsection{Subsection title}
%\label{sec:2}
%as required. Don't forget to give each section
%and subsection a unique label (see Sect.~\ref{sec:1}).
%\paragraph{Paragraph headings} Use paragraph headings as needed.

% For one-column wide figures use
%\begin{figure}
% Use the relevant command to insert your figure file.
% For example, with the graphicx package use
%  \includegraphics{example.eps}
% figure caption is below the figure
%\caption{Please write your figure caption here}
%\label{fig:1}       % Give a unique label
%\end{figure}
%
% For two-column wide figures use
%\begin{figure*}
% Use the relevant command to insert your figure file.
% For example, with the graphicx package use
  %\includegraphics[width=0.75\textwidth]{example.eps}
% figure caption is below the figure
%\caption{Please write your figure caption here}
%\label{fig:2}       % Give a unique label
%\end{figure*}
%

% Authors must disclose all relationships or interests that 
% could have direct or potential influence or impart bias on 
% the work: 
%

%\section*{Conflict of interest}
%
%I have no conflict of interest in submitting this manuscript. 

%\section*{Data availability statement}
%Data sharing not applicable to this article as no datasets were generated or analyzed during the current study.

\noindent Email: yandi.wu@wisc.edu \\
DEPARTMENT OF MATHEMATICS, UNIVERSITY OF WISCONSIN, MADISON 

\end{document}